\documentclass[11pt]{article}
\usepackage{amssymb}
\usepackage{amsmath}
\usepackage{amsthm} 
\usepackage{color}
\usepackage{wrapfig}
\usepackage{float}
\usepackage{epic}   
\usepackage{mathrsfs}

\allowdisplaybreaks[4] 

\makeatletter
\def\section{\@startsection{section}{1}{\z@}{-4.5ex plus -1ex minus
    -.2ex}{2.7ex plus .2ex}{\Large\bf}}

\newtheorem{definition}{Definition}
\newtheorem{lemma}{Lemma}
\newtheorem{remark}{Remark}
\newtheorem{theorem}{Theorem}
\numberwithin{equation}{section}

\renewenvironment{proof}{ 
\noindent{\bf Proof.} \rm}{\penalty-20\null\hfill $\square$}

\newcommand{\bbF}{{\mathbb F}}

\newcommand{\bbI}{{\mathbb I}}

\newcommand{\R}{{\mathbb R}}

\newcommand{\bbZ}{{\mathbb Z}}

\newcommand{\bfe}{\mathbf{e}}
\newcommand{\bff}{\mathbf{f}}
\newcommand{\bfg}{\mathbf{g}}
\newcommand{\bfh}{\mathbf{h}}

\newcommand{\bfk}{\mathbf{k}}

\newcommand{\bfn}{\mathbf{n}}

\newcommand{\bfu}{\mathbf{u}}
\newcommand{\bfv}{\mathbf{v}}
\newcommand{\bfw}{\mathbf{w}}
\newcommand{\bfx}{\mathbf{x}}

\newcommand{\bfz}{\mathbf{z}}

\newcommand{\bfphi}{\mbox{\boldmath $\phi$}}

\newcommand{\bfC}{\mathbf{C}}

\newcommand{\bfF}{\mathbf{F}}
\newcommand{\bfG}{\mathbf{G}}
\newcommand{\bfH}{\mathbf{H}}

\newcommand{\bfL}{\mathbf{L}}

\newcommand{\bfV}{\mathbf{V}}
\newcommand{\bfW}{\mathbf{W}}

\newcommand{\cA}{{\cal A}}

\newcommand{\cC}{{\cal C}}

\newcommand{\gB}{\mathfrak{B}}

\newcommand{\rmd}{\mathrm{d}}

\newcommand{\br}{\hspace{0.7pt}}
\renewcommand{\div}{\mathrm{div}\,}

\newcommand{\bfzero}{\mathbf{0}}
\newcommand{\dist}{\mathrm{dist}}

\newcommand{\supp}{\mathrm{supp}\,}

\newcommand{\Gammai}{\Gamma_{\hspace{-1.1pt} \rm in}}
\newcommand{\Gammao}{\Gamma_{\hspace{-1.1pt} \rm out}}
\newcommand{\Gammaw}{\Gamma_{\hspace{-1.4pt} \rm p}}
\newcommand{\Gammam}{\Gamma_{\hspace{-1pt} 0}}
\newcommand{\Gammap}{\Gamma_{\hspace{-1pt} 1}}
\newcommand{\gammai}{\gamma_{\rm in}}
\newcommand{\gammao}{\gamma_{\rm out}}
\newcommand{\Am}{A_0}
\newcommand{\Ap}{A_1}
\newcommand{\Bm}{B_0}
\newcommand{\Bp}{B_1}
\newcommand{\Omegah}{\widehat{\Omega}}

\newcommand{\bfgs}{\bfg_{\displaystyle *}}

\newcommand{\cCs}{\boldsymbol{\cC}^{\infty}_{\sigma}(\Omega)}
\newcommand{\cCns}{\boldsymbol{\cC}^{\infty}_{0,\sigma}(\Omega)}
\newcommand{\cCn}{\boldsymbol{\cC}^{\infty}_0(\Omega)}
\newcommand{\Vs}[1]{\bfV_{\sigma}^{1,{#1}}\hspace{-0.6pt}(\Omega)}
\newcommand{\Vsd}[1]{\bfV_{\sigma}^{-1,{#1}}\hspace{-0.6pt}(\Omega)}
\newcommand{\vs}[1]{\bfV_{\sigma}^{1,{#1}}}
\newcommand{\vsd}[1]{\bfV_{\sigma}^{-1,{#1}}}

\newcommand{\bfWp}{\bfW_{\rm per}}

\newcommand{\bfvt}{\widetilde{\bfv}}

\newcommand{\pt}{\widetilde{p}}
\newcommand{\bfft}{\widetilde{\bff}}

\newcommand{\hht}{\widetilde{h}}

\newcommand{\varphit}{\widetilde{\varphi}}
\newcommand{\psit}{\widetilde{\psi}}
\newcommand{\ft}{\widetilde{f}}
\newcommand{\operL}{\mathscr{L}}

\newcommand{\lclosed}{[\hbox to 1pt{}}
\newcommand{\rclosed}{\hbox to 1pt{}]}

\newcommand{\llangle}{\langle\hspace{-1.9pt}\langle}
\newcommand{\rrangle}{\rangle\hspace{-1.9pt}\rangle}
\newcommand{\blangle}{\bigl\langle}
\newcommand{\brangle}{\bigr\rangle}
\newcommand{\bllangle}{\blangle\hspace{-2.5pt}\blangle}
\newcommand{\brrangle}{\brangle\hspace{-2.5pt}\brangle}

\newcounter{constants}
\newcommand{\cn}[2]{ \addtocounter{constants}{1}
\newcounter{c#1#2}
\setcounter{c#1#2}{\value{constants}} c_{\arabic{c#1#2}} }
\newcommand{\cc}[2]{c_{\arabic{c#1#2}}}

\definecolor{lightgrey}{rgb}{0.82,0.82,0.82}

\begin{document}

\title{\LARGE \bf The weak Stokes problem associated with a flow through a
profile cascade in $L^r$-- framework}

\author{Tom\'a\v{s} Neustupa}

\date{}

\maketitle

\begin{abstract}
We study the weak steady Stokes problem, associated with a flow of
a Newtonian incompressible fluid through a spatially periodic
profile cascade, in the $L^r$--\br\br setup. The used mathematical
model is based on the reduction to one spatial period, represented
by a bounded 2D domain $\Omega$. The corresponding Stokes problem
is formulated by means of three types of boundary conditions: the
conditions of periodicity on the ``lower'' and ''upper'' parts of
the boundary, the Dirichlet boundary conditions on the ``inflow''
and on the profile and an artificial ``do nothing''--type boundary
condition on the ``outflow''. Under appropriate assumptions on the
given data, we prove the existence and uniqueness of a weak
solution in $\bfW^{1,r}(\Omega)$ and its continuous dependence on
the data. We explain the sense in which the ``do nothing''
boundary condition on the ``outflow'' is satisfied.
\end{abstract}

\noindent
{\it AMS math.~classification (2000):} \ 35Q30, 76D03, 76D05.


\section{Introduction} \label{S1}

{\bf One spatial period: domain $\Omega$.} \ Mathematical models
of a flow through a three--dimensi\-onal turbine wheel often use
the reduction to two space dimensions, where the flow is stu\-died
as a flow through an infinite planar profile cascade. In an
appropriately chosen Cartesian coordinate system, the profiles in
the cascade periodically repeat with the period $\tau$
\begin{wrapfigure}[14]{r}{81mm}
\begin{center}
  \setlength{\unitlength}{0.5mm}
  \begin{picture}(162,125)
  \put(5,60){\vector(1,0){151}} \put(20,37){\vector(0,1){95}}
  \put(143,64){$x_1$} \put(23,125){$x_2$}
  \put(28.5,106.4){\vector(0,1){3.2}}
  \dashline[+30]{2.2}(28.5,104.9)(28.5,75.2)
  \dashline[+30]{2.2}(130.1,70)(130.1,96)
  \dashline[+30]{2.2}(130.1,108)(130.1,132)
  \dashline[+30]{2.2}(130.1,40)(130.1,31)
  \put(118,100){$x_1=d$}
  \put(28.5,73.6){\vector(0,-1){3.7}} \put(31,97){$\tau$}
  \put(23,45){$\gammai$} \put(133,123){$\gammao$}
  \thicklines 
  \color{lightgrey}
\put(40.2,90){\line(1,0){35}} \put(40.2,90.4){\line(1,0){34}}
\put(40.2,90.8){\line(1,0){33}} \put(40.4,91.2){\line(1,0){31.6}}
\put(40.4,91.6){\line(1,0){30.4}} \put(40.5,92){\line(1,0){29.3}}
\put(40.5,92.4){\line(1,0){28.1}}\put(40.8,92.8){\line(1,0){26.5}}
\put(41.0,93.2){\line(1,0){25.0}}\put(41.5,93.6){\line(1,0){23.3}}
\put(41.8,94.0){\line(1,0){21.3}}\put(42.1,94.4){\line(1,0){19.3}}
\put(42.5,94.8){\line(1,0){17.3}}\put(43.2,95.2){\line(1,0){14.8}}
\put(44.3,95.6){\line(1,0){11.8}}\put(46.3,96.0){\line(1,0){6.8}}
\put(40.2,89.6){\line(1,0){36.1}}\put(40.4,89.2){\line(1,0){36.7}}
\put(40.6,88.8){\line(1,0){37.4}}\put(40.8,88.4){\line(1,0){38.2}}
\put(41.1,88.0){\line(1,0){38.9}}\put(41.5,87.6){\line(1,0){39.4}}
\put(42.1,87.2){\line(1,0){39.5}}\put(42.7,86.8){\line(1,0){39.8}}
\put(43.7,86.4){\line(1,0){39.7}}\put(44.5,86){\line(1,0){39.6}}
\put(45.7,85.6){\line(1,0){39.3}}\put(46.8,85.2){\line(1,0){39.1}}
\put(48.5,84.8){\line(1,0){38.2}}\put(50.4,84.4){\line(1,0){37.0}}
\put(52.3,84.0){\line(1,0){35.8}}\put(55.0,83.6){\line(1,0){33.8}}
\put(58.3,83.2){\line(1,0){31.2}}\put(61.7,82.8){\line(1,0){28.6}}
\put(65.3,82.4){\line(1,0){25.6}}\put(68.0,82.0){\line(1,0){23.5}}
\put(70.2,81.6){\line(1,0){21.9}}\put(72.2,81.2){\line(1,0){20.9}}
\put(74.2,80.8){\line(1,0){19.4}}\put(75.9,80.4){\line(1,0){18.3}}
\put(77.5,80.0){\line(1,0){17.3}}\put(79.0,79.6){\line(1,0){16.3}}

\put(80.8,79.2){\line(1,0){15.3}}\put(82.3,78.8){\line(1,0){14.4}}
\put(84.0,78.4){\line(1,0){13.3}}\put(85.5,78.0){\line(1,0){12.6}}
\put(87,77.6){\line(1,0){11.7}}\put(88.5,77.2){\line(1,0){10.8}}
\put(90.0,76.8){\line(1,0){9.9}}\put(91,76.4){\line(1,0){9.3}}
\put(92.2,76.0){\line(1,0){8.8}}\put(93.4,75.6){\line(1,0){8.2}}
\put(94.4,75.2){\line(1,0){7.4}}\put(95.8,74.8){\line(1,0){6.7}}
\put(97.1,74.4){\line(1,0){5.6}}\put(98.3,74.0){\line(1,0){4.6}}
\put(99.5,73.6){\line(1,0){3.5}}\put(100.3,73.2){\line(1,0){3.0}}
\color{black} 
\qbezier(40,90)(40,100)(63,94) \qbezier(40,90)(40,85)(60,83)
\qbezier(63,94)(86,87)(98,78) \qbezier(60,83)(80,80.5)(97.2,74.3)
\qbezier(98,78)(109,70)(97.2,74.3) \put(51,88){$P$}
  \thicklines 
  \put(20,69){\line(0,1){40}} \put(130.2,40){\line(0,1){40}}
  \qbezier(20,69)(70,76)(130,40) \qbezier(20,109)(70,116)(130,80)
  \thinlines
  \put(9.5,89){$\Gammai$} \put(116.2,67){$\Gammao$}
  \put(75,107){$\Gammap$} \put(75,67){$\Gammam$}
  \put(37,79){$\Gammaw$}
  \put(109,78){$\Omega$}
  \put(8,67){$\Am$} \put(8,107){$\Ap$}
  \put(133,39){$\Bm$} \put(133,78){$\Bp$}
  \thicklines
  \put(43,23){Fig.~1: \ Domain $\Omega$}
  \end{picture}
\end{center}
\end{wrapfigure}
in the $x_2$--direction. It can be natural\-ly assumed that the
flow is $\tau$--periodic in variable $x_2$, too. This enables one
to study the flow through one spatial period, which contains just
one profile -- see domain $\Omega$ and profile $P$ on Fig.~1. This
approach is used e.g.~in \cite{DFF} and \cite{KLP}, where the
authors present the numerical analysis of the models or
corresponding numerical simulations, and in the papers,
\cite{FeNe1}--\cite{FeNe3} and \cite{TNe1}--\cite{TNe4}, devoted
to theoretical analysis of the mathema\-tical models.

We assume that the fluid flows into the cascade through the
straight line $\gammai$ (the $x_2$--axis, the inflow) and
essentially leaves the cascade through the straight line
$\gammao$, whose equation is $x_1=d$ (the outflow). By
``essentially'' we mean that we do not exclude possible reverse
flows on $\gammao$. The considered spatial period $\Omega$ is
mainly determined by artificially chosen curves $\Gammam$ and
$\Gammap\equiv\Gammam+\tau\br\bfe_2$, which form the ``lower'' and
``upper'' parts of $\partial\Omega$ (the boundary of $\Omega$),
respectively. (See Fig.~1. We denote by $\bfe_2$ the unit vector
in the $x_2$--direction.) The parts of $\partial\Omega$, lying on
the straight lines $\gammai$ and $\gammao$, are the line segments
$\Gammai\equiv\Am\Ap$ and $\Gammao\equiv\Bm\Bp$, respectively, of
length $\tau$. The last part of $\partial\Omega$, i.e.~the
boundary of profile the $P$ is denoted by $\Gammaw$. We assume
that the domain $\Omega$ is Lipschitzian and the curves $\Gammam$
and $\Gammap$ are of the class $C^{\infty}$.

\vspace{4pt} \noindent
{\bf The Stokes boundary--value problem on one spatial period.} \
The flow of an incompressible Newtonian fluid is described by the
Navier-Stokes equations. An important role in theoretical studies
of these equations play the properties of solutions of the steady
Stokes problem. If one neglects the derivative with respect to
time and the nonlinear ``convective'' term in the momentum
equation in the Navier--Stokes system, one obtains
\begin{equation}
-\nu\Delta\bfu+\nabla p\ =\ \bff. \label{1.1}
\end{equation}
This equation is studied together with the equation of continuity
(= condition of incompressibility)
\begin{equation}
\div\bfu\ =\ 0 \label{1.2}
\end{equation}
and the equations (\ref{1.1}), (\ref{1.2}) represent the so called
{\it steady Stokes system}, or the {\it steady Stokes equations.}
The unknowns are $\bfu=(u_1,u_2)$ (the velocity) and $p$ (the
pressure). The positive constant $\nu$ is the kinematic
coefficient of viscosity and $\bff$ denotes the external body
force. The density of the fluid can be without loss of generality
supposed to be equal to one. The system (\ref{1.1}), (\ref{1.2})
is completed by appropriate boundary conditions on
$\partial\Omega$. One can naturally assume that the velocity
profile on $\Gammai$ is known, which leads to the inhomogeneous
Dirichlet boundary condition
\begin{equation}
\bfu\ =\ \bfg \qquad \mbox{on}\ \Gammai. \label{1.3}
\end{equation}
Further, we consider the homogeneous Dirichlet boundary condition
\begin{alignat}{3}
\bfu\ &=\ \bfzero && \mbox{on}\ \Gammaw, \label{1.4} \\
\noalign{\vskip 4pt \noindent the condition of periodicity on
$\Gammam$ and $\Gammap$ \vskip 5pt}
\bfu(x_1,x_2+\tau)\ &=\ \bfu(x_1,x_2) \qquad && \mbox{for}\
\bfx\equiv(x_1,x_2)\in\Gammam \label{1.5} \\
\noalign{\vskip 5pt \noindent and the artificial boundary
condition \vskip 4pt}
-\nu\, \frac{\partial\bfu}{\partial\bfn}+p\br\bfn\ &=\ \bfh &&
\mbox{on}\ \Gammao, \label{1.6}
\end{alignat}
where $\bfh$ is a given vector--function on $\Gammao$ and $\bfn$
denotes the unit outer normal vector, which is equal to
$\bfe_1\equiv(1,0)$ on $\Gammao$. The boundary condition
(\ref{1.6}) (with $\bfh=\bfzero$) is often called the ``do
nothing'' condition, because it naturally follows from an
appropriate weak formulation of the boundary--value problem, see
\cite{Glow} and \cite{HeRaTu}.

\vspace{4pt} \noindent
{\bf On the results of this paper.} \ The main results of the
paper are theorems on properties of so called weak Stokes operator
(Theorem \ref{T1}) and on the existence, uniqueness and continuous
dependence on the data of a weak solution $\bfu$ to the Stokes
problem (\ref{1.1})--(\ref{1.6}) in the $L^r$--setting (Theorem
\ref{T3}). The weak solution is defined in Definition \ref{D1}.
Theorem \ref{T2} provides the existence of an associated pressure
$p$ and explains the sense, in which $\bfu$ and $p$ satisfy the
boundary condition (\ref{1.6}). These results do not follow from
the previous cited papers on the Stokes problem, mainly because we
consider three different types of boundary conditions on
$\partial\Omega$, two of which ``meet'' in the ``corner points''
$\Am$, $\Ap$, $\Bm$ and $\Bp$ of domain $\Omega$. Moreover, while
the corresponding $L^2$--theory is relatively simple (see
\cite{TNe4}), the general $L^r$--case is much more difficult. The
key inequality (\ref{3.4}) is proven in Section \ref{S5}. The
crucial estimate of the $W^{1,r}$--norm of the velocity in the
neighborhood of $\Gammao$ is obtained, applying results of
S.~Agmon, A.~Douglis and L.~Nirenberg \cite{AgDoNi} on general
elliptic systems. As an auxiliary result, we present Lemma
\ref{L2} on an appropriate extension of the velocity profile
$\bfg$ from $\Gammai$ to $\Omega$.

\vspace{4pt} \noindent
{\bf On some previous related results.} \ The Stokes problem, in
various domains and with various boundary conditions, has already
been studied in many papers and books. As to weak solutions, the
$L^2$--existential theory and the proof of uniqueness of an
existing weak solution is relatively simple, because one works in
a Hilbert space and can apply the Riesz theorem. However, the
corresponding $L^r$--theory for a general $r\in(1,\infty)$ is much
more difficult. Nevertheless, also in the $L^r$--setup, results on
the existence and uniqueness of weak solutions of the Stokes
problem with Dirichlet's boundary condition for the velocity can
be found e.g.~in \cite{Ca}, \cite{Te} and \cite{Ga}, with Navier's
boundary condition in \cite{AcAmCoGh} and with the Navier--type
boundary condition in \cite{AmSe}. Fundamental estimates have been
basically obtained by means of the Stokes fundamental solution and
the corresponding Green tensor in \cite{Ca} (in 3D) and \cite{Ga},
respectively also the inf--sup condition in \cite{AcAmCoGh} and
\cite{AmSe}. The 2D case (with Dirichlet's boundary condition and
in a smooth domain) has also been solved in \cite{Te} by
expressing the velocity by means of a stream function, application
of operator $\nabla^{\perp}$ to the Stokes equation and the
results on the biharmonic boundary--value problem.

In studies of the Navier--Stokes equations in channels or profile
cascades with artificial boundary conditions on the outflow, many
authors use various modifications of condition (\ref{1.6}). (See
e.g.~\cite{BrFa}), \cite{FeNe1}--\cite{FeNe3} and
\cite{TNe1}--\cite{TNe3}.) The reason is that condition
(\ref{1.6}) does not enable one to control the amount of kinetic
energy in $\Omega$ in the case of a reverse flow on $\Gammao$.
Hence modifications are suggested so that one can derive an energy
inequality, and consequently prove the existence of weak
solutions. In papers \cite{KuSka} and \cite{Ku}, the authors use
the boundary condition on the outflow in connection with a flow in
a channel, and they prove the existence of weak solutions  of the
Navier--Stokes equations for ``small data''. Possible reverse
flows on the ``outflow'' of a channel are controlled by means of
additional conditions in \cite{KraNe1}, \cite{KraNe2},
\cite{KraNe3}, where the Navier--Stokes equations are replaced by
the Navier--Stokes variational inequalities.

The regularity up to the boundary of existing weak solutions
(stationary or time--depen\-dent) to the Navier--Stokes equations
with the boundary condition (\ref{1.6}) on a part of the boundary
has not been studied in literature yet. This is mainly because one
at first needs a deeper information on existence, uniqueness and
regularity in the $L^r$--framework for the corresponding steady
Stokes problem. There are, to our knowledge, only two papers which
bring an information on regularity of a solution of this Stokes
problem: 1) paper \cite{KuBe}, where the authors studied a flow in
a 2D channel $D$ of a special geometry, considering the
homogeneous Dirichlet boundary condition on the walls and
condition (\ref{1.6}) on the outflow, and proved that the velocity
is in $\bfW^{2-\beta,2}(D)$ for certain $\beta\in(0,1)$, provided
that $\bff\in\bfL^2(D)$ (see \cite[Theorem 2.1]{KuBe}), and 2)
paper \cite{TNe4}, where the belonging of the solution $(\bfu,p)$
of the Stokes problem (\ref{1.1})--(\ref{1.6}) to
$\bfW^{2,2}(\Omega)\times W^{1,2}(\Omega)$ has been recently
proven, under natural assumptions on $\bff$, $\bfg$ and $\bfh$.

\section{Notation and some preliminary results} \label{S2}

{\bf Notation.} \ Recall that $\Omega$ is a Lipschitzian domain in
$\R^2$, sketched on Fig.~1. Its boundary consists of the curves
$\Gammai$, $\Gammao$, $\Gammam$, $\Gammap$ and $\Gammaw$,
described in Section \ref{S1}. We denote by $\bfn=(n_1,n_2)$ the
outer normal vector field on $\partial\Omega$. We use $c$ as a
generic constant, i.e.~a constant whose values may change
throughout the text.

\begin{list}{$\circ$}
{\setlength{\topsep 0.5mm}
\setlength{\itemsep 0.5mm}
\setlength{\leftmargin 14pt}
\setlength{\rightmargin 0pt}
\setlength{\labelwidth 6pt}}

\item
$\Gammai^0$, respectively $\Gammao^0$, denotes the open line
segment without the end points $\Am,\, \Ap$, respectively $\Bm,\,
\Bp$. Similarly, $\Gammam^0$, respectively $\Gammap^0$ denotes the
curve $\Gammam$, respectively $\Gammap$, without the end points
$\Am$, $\Bm$, respectively $\Ap$, $\Bp$.

\item
We denote vector functions and spaces of vector functions by
boldface letters. If $\bfv\equiv(v_1,v_2)$ is a vector function
then $\nabla\bfv$ is a tensor function with the entries $\partial
v_i/\partial x_j\equiv\partial_jv_i$ on the positions $ij$ (for
$i,j=1,2$). Spaces of 2nd--order tensor functions are denoted by
the superscript $2\times 2$.

\item
We denote by $\|\, .\, \|_r$ the norm in $L^r(\Omega)$ or in
$\bfL^r(\Omega)$ or in $L^r(\Omega)^{2\times 2}$. Similarly, $\|\,
.\, \|_{s,r}$ is the norm in $W^{s,r}(\Omega)$ or in
$\bfW^{s,r}(\Omega)$ or in $W^{s,r}(\Omega)^{2\times 2}$.

\item
$\cCs$ denotes the linear space of all infinitely differentiable
divergence--free vector functions in $\overline{\Omega}$, whose
support is disjoint with $\Gammai\cup\Gammaw$ and that satisfy,
together with all their derivatives (of all orders), the condition
of periodicity (\ref{1.5}). Note that each $\bfw\in\cCs$
automatically satisfies the outflow condition
$\int_{\Gammao}\bfw\cdot\bfn\; \rmd l=0$.

\item
We denote by $\cCns$ the intersection $\cCs\cap\cCn$, where $\cCn$
is the space of all infinitely differentiable vector functions in
$\Omega$ with a compact support in $\Omega$.

\item
$\Vs{r}$ (for $r>1$) is the closure of $\cCs$ in
$\bfW^{1,r}(\Omega)$. The space $\Vs{r}$ can be characterized as a
space of divergence--free vector functions from
$\bfW^{1,r}(\Omega)$, whose traces on $\Gammai\cup\Gammaw$ are
equal to zero and the traces on $\Gammam$ and $\Gammap$ satisfy
the condition of periodicity (\ref{1.5}). Note that as functions
from $\Vs{r}$ are equal to zero on $\Gammai\cup\Gammaw$ (in the
sense of traces) and the domain $\Omega$ is bounded, the norm in
$\Vs{r}$ is equivalent to $\|\nabla.\, \|_r$.

\item
We denote by $r'$ the conjugate exponent to $r$, by
$\bfW^{-1,r}(\Omega)$ the dual space to $\bfW^{1,r'}_0(\Omega)$
and by $\bfW^{-1,r}_0(\Omega)$ the dual space to
$\bfW^{1,r'}(\Omega)$. The corresponding norms are denoted by
$\|\, .\, \|_{\bfW^{-1,r}}$ and $\|\, .\, \|_{\bfW^{-1,r}_0}$,
respectively.

\item
$\Vsd{r}$ denotes the dual space to $\Vs{r'}$. The duality pairing
between $\Vsd{r}$ and $\Vs{r'}$ is denoted by $\langle\, .\, ,\,
.\, \rangle_{(\vsd{r},\vs{r'})}$. The norm in $\Vsd{r}$ is denoted
by $\|\, .\, \|_{\vsd{r}}$.

\item
Denote by $\cA_r$ the linear mapping of $\Vs{r}$ to $\Vsd{r}$,
defined by the equation
\begin{equation}
\blangle\cA_r\bfv,\bfw\brangle_{(\vsd{r},\vs{r'})}\ :=\
(\nabla\bfv,\nabla\bfw) \qquad \mbox{for}\ \bfv\in\Vs{r}\
\mbox{and}\ \bfw\in\Vs{r'}, \label{2.1}
\end{equation}
where $(\nabla\bfv,\nabla\bfw)$ represents the integral
$\int_{\Omega}\nabla\bfv:\nabla\bfw\; \rmd\bfx$. We call operator
$\cA_r$ the {\it weak Stokes operator.}

\end{list}

\begin{lemma} \label{L4}
Let $\bff\in\bfW^{-1,r}_0(\Omega)$. Then there exists $\bbF\in
L^r(\Omega)^{2\times 2}$, satisfying $\div\bbF=\bff$ in the sense
of distributions in $\Omega$ and
\begin{equation}
\|\bbF\|_r\ \leq\ c\, \|\bff\|_{\bfW^{-1,r}_0}\, , \label{2.2}
\end{equation}
where $c$ is independent of $\bff$ and $\bbF$.
\end{lemma}

\begin{proof}
The proof is based on the results from paper \cite{GeHeHi} by
M.~Geussert, H.~Heck and M.~Hieber. If domain $\Omega'\subset\R^3$
is bounded and star-shaped with respect to some ball
$K\subset\Omega'$ and $\omega$ is a function in $C^{\infty}_0(K)$,
such that $\int_{K}\omega\, \rmd\bfx=1$ then it follows from
Proposition 2.1 in \cite{GeHeHi} that there is a bounded linear
operator $\gB:W^{-1,r}_0(\Omega')\to\bfL^r(\Omega')$, such that
$\div\gB f=f-\omega\int_{\Omega'}f\, \rmd\bfx$ for $f\in
L^r(\Omega')$. Applying an appropriate limit procedure, one can
show that if $f\in W^{-1,r}_0(\Omega')$ then
\begin{displaymath}
\div\gB f\ =\ f-\omega\, \langle f,1\rangle_{(W^{-1,r}_0(\Omega'),
W^{1,r'}(\Omega'))}
\end{displaymath}
in the sense of distributions. Then $\div[\gB f+\bfz]=f-\omega\,
\langle f,1\rangle_{(W^{-1,r}_0(\Omega'),W^{1,r'}(\Omega'))}+
\div\bfz$ for $\bfz\in \bfL^r(\Omega')$. Let us choose $\bfz$ so
that $\bfz=\bfz_1+\bfz_2$, where
\begin{equation}
\div\bfz_1\ =\ (\omega-\overline{\omega})\, \langle
f,1\rangle_{(W^{-1,r}_0(\Omega'),W^{1,r'} (\Omega'))}, \label{2.3}
\end{equation}
and $\bfz_2=\frac{1}{3}\bfx\, \overline{\omega}\, \langle
f,1\rangle_{(W^{-1,r}_0(\Omega'), W^{1,r'}(\Omega'))}$. (We denote
by $\overline{\omega}$ the mean value of $\omega$ in $\Omega'$.)
As the mean value of the right hand side of (\ref{2.3}) in
$\Omega'$ is zero, the equation (\ref{2.3}) is solvable in
$W^{1,r}_0(\Omega')$ due to \cite[Theorem 2.5]{GeHeHi}. Moreover,
\begin{displaymath}
\|\bfz_1\|_{1,r;\, \Omega'}\ \leq\ c\,
\|\omega-\overline{\omega}\|_{r;\, \Omega'}\, \bigl| \langle
f,1\rangle_{(W^{-1,r}_0(\Omega'),W^{1,r'} (\Omega'))} \bigr|\
\leq\ c\, \|f\|_{W^{-1,r}_0(\Omega)},
\end{displaymath}
where $c$ depends only on $\Omega'$, $\omega$ and $r$.
Furthermore,
\begin{displaymath}
\div\bfz_2=\overline{\omega}\, \langle
f,1\rangle_{(W^{-1,r}_0(\Omega'), W^{1,r'}(\Omega'))} \qquad
\mbox{and} \qquad \|\bfz_2\|_{1,r;\, \Omega'}\ \leq\ c\,
\|f\|_{W^{-1,r}_0(\Omega)},
\end{displaymath}
Thus, the function $\gB f+\bfz$ satisfies the equation $\div[\gB
f+\bfz]=f$ in the sense of distributions in $\Omega'$ and $\|\gB
f+\bfz\|_{r;\, \Omega'}\leq c\, \|f\|_{W^{-1,r}_0(\Omega')}$.

Domain $\Omega$ can be expressed as a finite union of star-shaped
domains. This enables us to carry over these results to the whole
domain $\Omega$, applying the same arguments as in \cite{GeHeHi},
pp.~116--117. Thus, we can formulate the proposition: {\it there
exists a bounded linear operator
$\widetilde{\gB}:W^{-1,r}_0(\Omega)\to\bfL^r(\Omega)$, such that
$\div\widetilde{\gB}f=f$ in the sense of distributions in $\Omega$
for $f\in W^{-1,r}_0(\Omega)$.} Now, it is just a technical step
to extend this proposition from $f\in W^{-1,r}_0(\Omega)$ to
$\bff\in\bfW^{-1,r}_0(\Omega)$. This completes the proof.
\end{proof}

\begin{remark} \label{R1} \rm
As there is not a complete coincidence on the definition of the
divergence of a tensor field in literature, note that if
$\bbF=(F_{ij})$ ($i,j=1,2$) then $\div\bbF$ in Lemma \ref{L4} (and
also further on throughout the paper) denotes the vector
$\partial_jF_{ij}$ ($i=1,2$). In accordance with this notation,
$\div\nabla\bfv$ is the vector with the entries
$\partial_j(\partial_jv_i)=\Delta v_i$ ($i=1,2$).

One can also assume that $\bff\in\bfW^{-1,r}(\Omega)$ (instead of
$\bff\in\bfW^{-1,r}_0(\Omega)$) in Lemma \ref{L4}. However, in
this case, one cannot apply \cite{GeHeHi} in order to obtain the
tensorial function $\bbF\in L^r(\Omega)^{2\times 2}$ with the
properties stated in the lemma. Nevertheless, the existence of
$\bbF$, satisfying the equation $\div\bbF=\bff$ (in the sense of
distributions) and the estimate $\|\bbF\|_r\leq c\,
\|\bff\|_{\bfW^{-1,r}}$ can be proven in this case, too, just
appropriately modifying the proof of Lemma II.1.6.1 in \cite{So},
which concerns the case $r=2$.
\end{remark}

\section{The weak Stokes operator} \label{S3}

The next theorem is an analogue of results, known on the Stokes
problem with the homogeneous Dirichlet or Navier or Navier--type
boundary conditions on the whole boundary of domain $\Omega$, see
\cite{Ga}, \cite{AcAmCoGh} and \cite{AmSe}. Recall that the
theorem is non-trivial especially due to the variety of used
boundary conditions and the fact that one cannot apply Riesz'
theorem in the general $L^r$--framework in order to establish the
existence and uniqueness of a solution of the equation
$\cA_r\bfv=\bff$ for $\bff\in\Vsd{r}$.

\begin{theorem}[on the weak Stokes operator $\cA_r$] \label{T1}
The weak Stokes operator $\cA_r$ is a boun\-ded, closed and
injective operator from $\Vs{r}$ to $\Vsd{r}$ with
$D(\cA_r)=\Vs{r}$ and $R(\cA_r)=\Vsd{r}$. The adjoint operator to
$\cA_r$ is $\cA_r^*=\cA_{r'}$.
\end{theorem}

\begin{proof}
The case $r=2$ is proven in \cite{TNe4}. Thus, assume that
$r\not=2$. We split the proof to several parts.

(a) \ Denote by $\|\cA_r\|_{\vs{r}\to\vsd{r}}$ the norm of
operator $\cA_r$. Since
\begin{align*}
\|\cA_r\|_{\vs{r}\to\vsd{r}}\ &=\ \sup_{\bfv\in\Vs{r},\
\bfv\not=\bfzero}\ \frac{\|\cA_r\bfv\|_{\vsd{r}}}{\|\bfv\|_{1,r}}\
\leq\ c\ \sup_{\bfv\in\Vs{r},\ \bfv\not=\bfzero}\
\frac{\|\cA_r\bfv\|_{\vsd{r}}}{\|\nabla\bfv\|_r} \\
&=\ c\ \sup_{\bfv\in\Vs{r},\ \bfv\not=\bfzero}\
\sup_{\bfw\in\Vs{r'},\ \bfw\not=\bfzero}\
\frac{|(\nabla\bfv,\nabla\bfw)|}{\|\nabla\bfv\|_r\,
\|\nabla\bfw\|_{r'}}\ \leq\ c,
\end{align*}
the operator $\cA_r$ is bounded. The identity $D(\cA_r)=\Vs{r}$
follows from the definition of $\cA_r$. Operator $\cA_r$ is
closed, as a bounded linear operator, defined on the whole space
$\Vs{r}$.

(b) \ In this part, we consider $\bfF\in\Vsd{2}$ of a special form
and deal with the equation $\cA_2\bfv=\bfF$. Concretely, we assume
that $\bbF\in W^{1,2}(\Omega)^{2\times 2}$ satisfies the condition
\begin{equation}
\bbF(x_1,x_2+\tau)\ =\ \bbF(x_1,x_2) \qquad \mbox{for a.a.}\
(x_1,x_2)\in\Gammam \label{3.1}
\end{equation}
and define $\bfF\in\Vsd{2}$ by the formula
\vspace{-4pt}
\begin{equation}
\blangle\bfF,\bfw\brangle_{(\vsd{2},\vs{2})}\ :=\
-\int_{\Omega}\bbF:\nabla\bfw\; \rmd\bfx \label{3.2}
\end{equation}
for all $\bfw\in\Vs{2}$. Then, due to \cite[Lemma 1 and Theorem
2]{TNe4}, the equation $\cA_2\bfv=\bfF$ has a unique solution
$\bfv\in\Vs{2}\cap\bfW^{2,2}(\Omega)$, there exists $p\in
W^{1,2}(\Omega)$ (an associated pressure), such that
$\nabla\bfv+p\br\bbI-\bbF\in W^{1,2}(\Omega)^{2\times 2}$, and
$\bfv$, $p$ satisfy the equation
\begin{equation}
-\Delta\bfv+\nabla p-\div\bbF\ \equiv\
\div(-\nu\nabla\bfv+p\br\bbI-\bbF)\ =\ \bfzero \label{3.3}
\end{equation}
a.e.~in $\Omega$. If $r>2$ then there exists $\cn01>0$,
independent of $\bfv$ and $\bbF$, such that
\begin{equation}
\|\bfv\|_{1,r}\ \leq\ \cc01\, \|\bbF\|_r. \label{3.4}
\end{equation}
This estimate is the key part of the proof of Theorem \ref{T1}. As
the proof of (\ref{3.4}) is relatively long, the used technique
differs from other sections, and in order not to disrupt the
logical sequence of the text, we postpone the derivation of
(\ref{3.4}) to a separate section (Section \ref{S5}).

(c) \ In this part, we assume that $r>2$, $\bbF\in
L^r(\Omega)^{2\times 2}$ and the functional $\bfF\in\Vsd{r}$ is
defined by the same formula as (\ref{3.2}) (for all
$\bfw\in\Vs{r'}$). We prove the solvability of the equation
$\cA_r\bfv=\bfF$.

There exists a sequence $\{\bbF_n\}$ in $W^{1,2}(\Omega)^{2\times
2}$, satisfying (\ref{3.1}), such that $\bbF_n\to\bbF$ in
$L^r(\Omega)^{2\times 2}$. Let $\{\bfF_n\}$ be a sequence of
functionals from $\Vsd{r}$, related to $\bbF_n$ through formula
(\ref{3.2}), which is now valid for all $\bfw\in\Vs{r'}$. The
functionals $\bfF$ and $\bbF_n$ satisfy
\vspace{-10pt}
\begin{equation}
\|\bfF_n-\bfF\|_{\vsd{r}}\ =\ \sup_{\bfw\in\cCs;\
\bfw\not=\bfzero}\ \|\bfw\|_{1,r'}^{-1}\
\biggl|\int_{\Omega}(\bbF_n-\bbF):\nabla\bfw\; \rmd\bfx\biggr|\
\leq\ c\, \|\bbF_n-\bbF\|_r. \label{3.5}
\end{equation}
Since $\Vsd{r}\subset\Vsd{2}$, there exists (due to part (b) of
this proof) a sequence $\{\bfv_n\}$ in
$\Vs{2}\cap\bfW^{2,2}(\Omega)$, such that $\cA_2\bfv_n=\bfF_n$ and
\begin{equation}
\|\bfv_n\|_{1,r}\ \leq\ \cc01\, \|\bbF_n\|_r. \label{3.6}
\end{equation}
As the space $\Vs{r}$ is reflexive, there exists a subsequence
(which we again denote by $\{\bfv_n\}$), such that
$\bfv_n\rightharpoonup\bfv$ in $\Vs{r}$. The functions $\bfv_n$
satisfy
\begin{displaymath}
(\nabla\bfv_n,\nabla\bfw)\ =\ \langle\bfF_n,\bfw\rangle_{(\vsd{2},
\vs{2})} \qquad \mbox{for all}\ \bfw\in\Vs{2}.
\end{displaymath}
If $\bfw\in\Vs{r'}$ then the right hand side equals
$\langle\bfF_n,\bfw\rangle_{(\vsd{r}, \vs{r'})}$. Thus, as
$\Vs{2}$ is dense in $\Vs{r'}$, we also have
\begin{displaymath}
(\nabla\bfv_n,\nabla\bfw)\ =\ \langle\bfF_n,\bfw\rangle_{(\vsd{r},
\vs{r'})} \qquad \mbox{for all}\ \bfw\in\Vs{r'}.
\end{displaymath}
The limit transition for $n\to\infty$ yields
\begin{displaymath}
(\nabla\bfv,\nabla\bfw)\ =\ \langle\bfF,\bfw\rangle_{(\vsd{r},
\vs{r'})} \qquad \mbox{for all}\ \bfw\in\Vs{r'},
\end{displaymath}
which means that $\cA_r\bfv=\bfF$. It also follows from the limit
transition that the function $\bfv$ satisfies inequality
(\ref{3.4}).

(d) \ Here, we show that each functional $\bff\in\Vsd{r}$ (for any
$r\in(1,\infty)$) can be expressed in the form (\ref{3.2}) for
some appropriate $\bbF\in L^r(\Omega)^{2\times 2}$. This, together
with part (c), shows that if $r>2$ and $\bff\in\Vsd{r}$ then the
equation $\cA_r\bfv=\bff$ is solvable. In other words,
$R(\cA_r)=\Vsd{r}$.

Thus, let $\bff\in\Vsd{r}$. As $\Vs{r'}$ is a closed subspace of
$\bfW^{1,r'}(\Omega)$, the functional $\bff$ can be extended (by
the Hahn--Banach theorem) to a bounded linear functional
$\bff_*\in\bfW^{-1,r}_0(\Omega)$, such that
$\|\bff_*\|_{\bfW^{-1,r}_0}=\|\bff\|_{\vsd{r}}$. There exists (by
Lemma \ref{L4}) $\bbF\in L^r(\Omega)^{2\times 2}$, such that
$\div\bbF=\bff_*$ in the sense of distributions in $\Omega$ and
$\|\bbF\|_r\leq c\,
\|\bff_*\|_{\bfW^{-1,r}_0}\equiv\|\bff\|_{\vsd{r}}$. As
$\bff_*=\bff$ on $\Vs{r'}$, we have
\vspace{-10pt}
\begin{displaymath}
\llangle \div\bbF,\bfw \rrangle\ =\ \llangle \bff,\bfw \rrangle
\qquad \forall\, \bfw\in\cCns,
\end{displaymath}
where $\llangle\, .\, ,\, .\, \rrangle$ denotes the action of a
distribution on a function from $\cCn$. Since the left hand side
is equal to
\vspace{-8pt}
\begin{displaymath}
-\llangle\bbF,\nabla\bfw\rrangle\ =\
-\int_{\Omega}\bbF:\nabla\bfw\; \rmd\bfx,
\end{displaymath}

\vspace{-8pt} \noindent
we obtain
\vspace{-6pt}
\begin{displaymath}
\hspace{20pt} \langle\bff,\bfw\rangle_{(\vsd{r},\vs{r'})}\ =\
\llangle\bff, \bfw\rrangle\ =\ -\int_{\Omega}\bbF:\nabla\bfw\;
\rmd\bfx \qquad \mbox{for all}\ \bfw\in\cCns.
\end{displaymath}
Both the left hand-- and right hand--sides can be continuously
extended so that they equal each other for all $\bfw\in\Vs{r'}$.
This means that $\bff=\bfF$, where $\bfF$ is related to $\bbF$
through formula (\ref{3.2}) (where we consider $\bfw\in\Vs{r'}$).
It follows from part (c) that there exists $\bfv\in\Vs{r}$, such
that $\cA_r\bfv=\bff$. There exists $c>0$, independent of $\bff$
and $\bfv$, such that the solution satisfies the estimate
\begin{equation}
\|\bfv\|_{1,r}\ \leq\ c\, \|\bff\|_{\vsd{r}}. \label{3.7}
\end{equation}

(e) \ In this part, we derive information on the adjoint operator
$\cA_r^*$ to $\cA_r$ for any $r\in(1,\infty)$. The adjoint
operator acts from $\Vsd{r}^*\equiv\Vs{r'}^{**}$ to $\Vs{r}^*=
\Vsd{r'}$. However, as $\Vs{r'}$ is reflexive, $\Vs{r'}^{**}$ can
be identified with $\Vs{r'}$. Thus, $\cA_r^*$ is an operator from
$\Vs{r'}$ to $\Vsd{r'}$. The domain of $\cA_r^*$ is, by
definition, the set of all $\bfw\in\Vs{r'}$, such that $\langle
\cA_r\bfz,\bfw\rangle_{(\vsd{r},\vs{r'})}$ is, in dependence on
$\bfz$, a bounded linear functional on $\Vs{r}$. This functional
is exactly $\cA_r^*\bfw$ and it satisfies
\begin{displaymath}
\langle\cA_r^*\bfw,\bfz\rangle_{(\vsd{r'},\vs{r})}=\langle
\cA_r\bfz,\bfw\rangle_{(\vsd{r},\vs{r'})}.
\end{displaymath}
By definition of $\cA_r$, the right hand side equals
$(\nabla\bfz,\nabla\bfw)$. This can be also written as
$(\nabla\bfw,\nabla\bfz)$ and it is, in dependence on $\bfz$, a
bounded linear functional acting on $\Vs{r}$ for each fixed
$\bfw\in\Vs{r'}$. From this, we deduce that $D(\cA_r^*)=\Vs{r'}$
and $\cA_{r}^*=\cA_{r'}$.

(f) \ Here, we assume that $1<r<\infty$ and prove the uniqueness
of the solution $\bfv$ of the equation $\cA_r\bfv=\bff$ for any
$\bff\in\Vsd{r}$.

Assume at first that $r>2$. Then, as $\Vs{r}\subset\Vs{r'}$, one
can use equation (\ref{2.1}) with $\bfw=\bfv$ and deduce that
$\cA_r\bfv=\bfzero$ implies $\bfv=\bfzero$, which means that the
solution of the equation $\cA_r\bfv=\bff$ is unique.

Assume now that $1<r<2$. The null space of $\cA_r$ and range of
$\cA_{r'}$ satisfy the identity $N(\cA_r)=R(\cA_{r'})^{\perp}$,
see \cite[p.~168]{Ka}. However, as $R(\cA_{r'})$ is the whole
space $\Vsd{r'}$ (because $r'>2$ and due to parts (c) and (d) of
this proof), the space of annihilators $R(\cA_{r'})^{\perp}$ is
trivial. Thus, we obtain the identity $N(\cA_r)=\{\bfzero\}$. This
implies the uniqueness of the solution $\bfv$ and the injectivity
of operator $\cA_r$.

(g) \ Finally, we assume that $1<r<2$ and prove the existence of a
solution $\bfv$ of the equation $\cA_r\bfv=\bff$ for any
$\bff\in\Vsd{r}$. Since $R(\cA_r)^{\perp}=N(\cA_{r'})=
\{\bfzero\}$, $R(\cA_r)$ is dense in $\Vsd{r}$. The operator
$\cA_{r'}^{-1}$ is bounded from $\Vsd{r'}$ to $\Vs{r'}$, which
follows from the inequality $r'>2$, parts (c) and (d) of this
proof and the closed graph theorem. Consequently, due to
\cite[p.~169]{Ka}, the operator $(\cA_{r'}^*)^{-1}\equiv
\cA_r^{-1}$ is bounded from $\Vsd{r}$ to $\Vs{r}$. Hence $\cA_r$
maps a closed set in $\Vs{r}$ onto a closed set in $\Vsd{r}$. It
means that $R(\cA_r)$ is closed in $\Vsd{r}$. Since it is also
dense in $\Vsd{r}$, we have $R(\cA_r)=\Vsd{r}$. This completes the
proof.
\end{proof}


\section{The weak Stokes problem} \label{S4}

\noindent
{\bf The spaces $\bfWp^{1-1/r,r}(\Gammai)$,
$\bfWp^{1-1/r',r'}(\Gammao)$ and $\bfWp^{-1/r,r}(\Gammao)$.} \
Recall that the straight lines $\gammai$, $\gammao$ and the line
segments $\Gammai$, $\Gammao$ are sketched on Fig.~1. Let
$1<r<\infty$. We denote by $\bfWp^{1-1/r,r}(\gammai)$ the space of
$\tau$--periodic functions in $W^{1-1/r,r}_{loc}(\gammai)$ and by
$\bfWp^{1-1/r,r}(\Gammai)$ the space of functions from
$W^{1-1/r,r}(\Gammai)$, that can be extended from $\Gammai$ to
$\gammai$ as functions in $\bfWp^{1-1/r,r}(\gammai)$.

The space $\bfWp^{1-1/r',r'}(\Gammao)$ is defined by analogy with
$\bfWp^{1-1/r,r}(\Gammai)$. Let us denote by
$\bfWp^{-1/r,r}(\Gammao)$ the dual space to
$\bfWp^{1-1/r',r'}(\Gammao)$ and by $\langle\, .\, ,\, .\,
\rangle_{(\bfWp^{-1/r,r}(\Gammao),\bfWp^{1-1/r',r'}(\Gammao))}$
the duality pairing between $\bfWp^{-1/r,r}(\Gammao)$ and
$\bfWp^{1-1/r',r'}(\Gammao)$.

An alternative description of $\bfWp^{1-1/r,r}(\Gammai)$ (which
also holds for $\bfWp^{-1/r,r}(\Gammao)$) is explained in the
following remark.

\begin{remark} \label{R3} \rm
Assume at first that $r>2$. In this case, functions from
$\bfWp^{1-1/r,r}(\Gammai)$ have traces at the end points $\Am$ and
$\Ap$ of $\Gammai$. Denote $\widetilde{\bfW}^{1-1/r,r}_{\rm per}
(\Gammai):= \{\bfw\in\bfW^{1-1/r,r}(\Gammai);\ \bfw(\Am)=
\bfw(\Ap)\ \mbox{in the sense of traces}\}$. Obviously,
$\bfWp^{1-1/r.r}(\Gammai)\subset\widetilde{\bfW}^{1-1/r,r}_{\rm
per}(\Gammai)$. Let us show that the opposite inclusion is also
true. Thus, let $\bfw\in\widetilde{\bfW}^{1-1/r,r}_{\rm
per}(\Gammai)$. Denote by the same symbol $\bfw$ the
$\tau$--periodic extension from $\Gammai$ to $\gammai$. We can
assume without loss of generality that $\Am=(0,0)$ and
$\Ap=(0,\tau)$. Put $A_2:=(0,2\tau)$. In order to verify that
$\bfw\in\bfW^{1-1/r,r}_{loc}(\gammai)$, it is sufficient to show
that $\bfw$ lies in the space $\bfW^{1-1/r,r}(\Am,A_2)$. The norm
of $\bfw$ in $\bfW^{1-1/r,r}(\Am A_2)$ equals $\|\bfw\|_{r;\, \Am
A_2}+ \llangle\bfw\rrangle_{1-1/r,r;\, \Am A_2}$ (see formulas (2)
and (3) in \cite{KuJoFu}, paragraph 8.3.2, p.~386), where
\begin{align*}
& \llangle\bfw\rrangle_{1-1/r,r;\, \Am A_2}\ =\ \int_0^{2\tau}\!
\int_0^{2\tau} \frac{|\bfw(0,y)-\bfw(0,z)|^r}{|y-z|^r}\; \rmd y\,
\rmd z \\
& \hspace{15pt} =\ \biggl(
\int_0^{\tau}\!\!\int_0^{\tau}\!+\int_0^{\tau}\!\!
\int_{\tau}^{2\tau} \!+\int_{\tau}^{2\tau}\!\!\int_0^{\tau}\!+
\int_{\tau}^{2\tau}\!\! \int_{\tau}^{2\tau} \biggr)\,
\frac{|\bfw(0,y)-\bfw(0,z)|^r}{|y-z|^r}\; \rmd y\, \rmd z \\
& \hspace{15pt} \leq\ \llangle\bfw\rrangle_{1-1/r,r;\,
\Am\Ap}+\llangle\bfw\rrangle_{1-1/r,r;\, \Ap A_2}+2\int_0^{\tau}
\!\!\int_{\tau}^{2\tau}\frac{|\bfw(0,y)-\bfw(0,z)|^r} {|y-z|^r}\;
\rmd y\, \rmd z.
\end{align*}
The first two terms on the right hand side are equal due to the
$\tau$--periodicity of $\bfw(0,\, .\, )$. The integral on the
right hand side is less than or equal to
\begin{align*}
c\int_0^{\tau} \!\!\int_{\tau}^{2\tau} & \biggl(
\frac{|\bfw(0,y)-\bfw(0,\tau)|^r}{(y-z)^r}+
\frac{|\bfw(0,\tau)-\bfw(0,z)|^r}{(y-z)^r} \biggr)\; \rmd y\,
\rmd z \\
&\leq\ c\int_{\tau}^{2\tau} \frac{|\bfw(0,y)-\bfw(0,\tau)|^r}
{(y-\tau)^{r-1}}\, \rmd y+c\int_0^{\tau}
\frac{|\bfw(0,\tau)-\bfw(0,z)|^r}{(\tau-z)^{r-1}}\,
\rmd z \\
&\leq\ c\int_{\tau}^{2\tau}\!\!\int_{\tau}^{2\tau}
\frac{|\bfw(0,y)-\bfw(0,x)|^r} {|y-x|^r}\, \rmd x\, \rmd y+
c\int_{0}^{\tau}\!\!\int_{0}^{\tau} \frac{|\bfw(0,x)-\bfw(0,z)|^r}
{|x-z|^r}\, \rmd x\, \rmd z.
\end{align*}
The last estimate holds due to the fractional Hardy inequality,
see e.g.~\cite[Theorem 2.1]{HeKuPe}. The right hand side is, due
to the $\tau$--periodicity of the function $\bfw(0,\, .\, )$, less
than or equal to $c\, \llangle\bfw\rrangle_{1-1/r,r;\, \Am\Ap}$.
This confirms that $\bfw\in\bfW^{1-1/r,r}(\Am,A_2)$, indeed.

In the critical case $r=2$, one cannot characterize
$\bfWp^{1/2,2}(\Gammai)$ as in the case $r>2$, because the traces
at the end points $\Am$ and $\Ap$ generally do not exist.
Moreover, although
$\bfW^{1/2,2}_0(\Gammai)=\bfW^{1/2,2}(\Gammai)$, see \cite[Theorem
II.11.1]{LiMa}), one cannot identify $\bfWp^{1/2,2}(\Gammai)$ with
$\bfW^{1/2,2}(\Gammai)$, because e.g.~the linear function
$g(0,y):=y$ for $(0,y)\in\Gammai$ is in $\bfW^{1/2,2}(\Gammai)$,
but its $\tau$--periodic extension to $\gammai$ is not in
$\bfW^{1/2,2}_{loc}(\gammai)$. Thus, one only has the inclusion
$\bfWp^{1/2,2}(\Gammai)\subset \bfW^{1/2,2}(\Gammai)$.

If $1<r<2$ then $\bfW^{1-1/r,r}_0(\Gammai)=
\bfW^{1-1/r,r}(\Gammai)$ (which follows from the density of
$\bfW^{1-1/r,r}(\Gammai)$ in $\bfW^{1/2,2}(\Gammai)$, the identity
$\bfW^{1/2,2}(\Gammai)=\bfW^{1/2,2}_0(\Gammai)$ and the density of
$\bfW^{1/2,2}_0(\Gammai)$ in $\bfC_0^{\infty}(\Gammai^0)$). This
and similar estimates of $\llangle\bfw\rrangle_{1-1/r,r;\, \Am
A_2}$, as above, enable one to show that every function from
$\bfW^{1-1/r,r}(\Gammai)$, periodically extended to $\gammai$ with
the period $\tau$, is in $\bfW^{1-1/r,r}_{loc}(\gammai)$. Thus,
one obtains the identity $\bfWp^{1-1/r,r}(\Gammai)=
\bfW^{1-1/r,r}(\Gammai)$.
\end{remark}

\begin{definition}[weak solution of the Stokes problem
(\ref{1.1})--(\ref{1.6})] \label{D1} \rm Let $r\in(1,\infty)$. Let
$\bff\in\bfW^{-1,r}_0(\Omega)$, $\bfg\in\bfW^{1-1/r,r}_{\rm
per}(\Gammai)$ and $\bfh\in\bfWp^{-1/r,r}(\Gammao)$. Let $\bbF\in
L^r(\Omega)^{2\times 2}$ be a tensor function, satisfying
$\div\bbF=\bff$ in the sense of distributions, provided by Lemma
\ref{L4}. A divergence--free function $\bfu\in
\bfW^{1,r}(\Omega)$, satisfying the conditions (\ref{1.3}),
(\ref{1.4}) and (\ref{1.5}) in the sense of traces on $\Gammai$,
$\Gammaw$ and $\Gammam$, respectively, and the equation
\begin{equation}
\nu\int_{\Omega}\nabla\bfu:\nabla\bfw\; \rmd\bfx\ =\
-\int_{\Omega}\bbF:\nabla\bfw\; \rmd\bfx+\langle\bfh,
\bfw\rangle_{(\bfWp^{-1/r,r}(\Gammao),\bfWp^{1-1/r',r'}(\Gammao))}
\label{4.1}
\end{equation}
for all $\bfphi\in\Vs{r'}$, is said to be a {\it weak solution} to
the Stokes problem (\ref{1.1})--(\ref{1.6}).
\end{definition}

We show in Theorem \ref{T2} that if $\bfu$ is a weak solution of
the problem (\ref{1.1})--(\ref{1.6}) then there exists an
associate pressure $p$ so that $\bfu$ and $p$ satisfy an analogue
of the condition (\ref{1.6}) in a certain weak sense on $\Gammao$.

\begin{theorem}[a posteriori properties of a weak solution] \label{T2}
1) Let $r$, $\bff$, $\bbF$, $\bfg$ and $\bfh$ satisfy the
assumptions from Definition \ref{D1}. Let $\bfu$ be a weak
solution of the Stokes problem (\ref{1.1})--(\ref{1.6}). Then
there exists an associated pressure $p\in L^r(\Omega)$ such that
$\bfu$ and $p$ satisfy the equations (\ref{1.1}) and (\ref{1.2})
in the sense of distributions in $\Omega$ and
\begin{equation}
(-\nu\br\nabla\bfu-p\br\bbI-\bbF)\cdot\bfn\ =\ \bfh
\label{4.3}
\end{equation}
holds as an equality in $\bfWp^{-1/r,r}(\Gammao)$.

2) If, moreover, $\bff\in\bfL^r(\Omega)$ then the tensor function
$\bbF$ can be constructed so that it lies in $W^{1,r}_{\rm
per}(\Omega)^{2\times 2}$, satisfies the equation $\div\bbF=\bff$
a.e.~in $\Omega$ and the condition $\bbF\cdot\bfn=\bfzero$ a.e.~on
$\Gammao$. In this case, (\ref{4.3}) takes the form
\begin{equation}
(-\nu\br\nabla\bfu-p\br\bbI)\cdot\bfn\ =\ \bfh, \label{4.4}
\end{equation}
consistent with (\ref{1.6}).
\end{theorem}

\begin{proof}
Suppose at first that the test function $\bfw$, used in
(\ref{4.1}), is in $\cCns$. Then, since $\bfw=\bfzero$ on
$\Gammao$, (\ref{4.1}) and the equation $\bff=\div\bbF$ imply that
\begin{displaymath}
\bllangle\nu\Delta\bfu+\div\bbF,\bfw\brrangle\ =\ 0.
\end{displaymath}
As this holds for all $\bfw\in\cCns$, we can apply De Rham's lemma
(see \cite[p.~14]{Te}) and deduce that there exists a distribution
$p_0$ in $\Omega$, such that
\begin{equation}
\nu\Delta\bfu+\div\bbF\ =\ \nabla p_0 \label{4.5}
\end{equation}
holds in $\Omega$ in the sense of distributions. As both
$\nu\Delta\bfu$ and $\div\bbF$ can also be identified with
elements of $\bfW^{-1,r}(\Omega)$, $\nabla p_0$ belongs to
$\bfW^{-1,r}(\Omega)$, too. It follows from \cite[Lemma
IV.1.1]{Ga} that $p_0\in L^r(\Omega)$ and it can be chosen so that
\begin{equation}
\|p_0\|_r\ \leq\ c\, \bigl\| \nu\Delta\bfu+\div\bbF
\bigr\|_{\bfW^{-1,r}}, \label{4.6}
\end{equation}
where $c$ depends only on $\nu$ and $\Omega$.

Since $-\nu\br\nabla\bfu+p_0\br\bbI-\bbF\in L^r(\Omega)^{2\times
2}$ and $\div(-\nu\br\nabla\bfu+p_0\br\bbI-\bbF)=\bfzero$ in the
sense of distributions, we can apply Theorem III.2.2 from
\cite{Ga} and deduce that
$(-\nu\br\nabla\bfu+p_0\br\bbI-\bbF)\cdot\bfn\in
\bfWp^{-1/r,r}(\Gammao)$ (in the sense of traces).

Let $\bfw\in\cCs$. The equation (\ref{4.1}) and the generalized
Gauss identity (see \cite[p.~160]{Ga} imply that
\begin{align*}
0\ &=\ \blangle\div[\nu\nabla\bfu+\bbF-p_0\br\bbI\br],
\bfw\brangle_{(\bfW^{-1,r}_0(\Omega),\bfW^{1,r'}(\Omega))} \\
&=\ \blangle (\nu\br\nabla\bfu+\bbF-p_0\br\bbI)\cdot\bfn,\bfw
\brangle_{(\bfW^{-1/r,r}(\partial\Omega),\bfW^{1-1/r',r'}(\partial\Omega))}
-\int_{\Omega}[\nu\br\nabla\bfu+\bbF]:\nabla\bfw\; \rmd\bfx \\
&=\ \blangle (\nu\br\nabla\bfu+\bbF-p_0\br\bbI)\cdot\bfn,\bfw
\brangle_{\bfWp^{-1/r,r}(\Gammao),\bfWp^{1-1/r',r'}(\Gammao)}-
\langle\bfh,\bfw\rangle_{(\bfWp^{-1,r}
(\Gammao),\bfWp^{1-1/r',r'}(\Gammao))} \\
& \hspace{15pt} +\langle\bfh,\bfw\rangle_{(\bfWp^{-1,r}
(\Gammao),\bfWp^{1-1/r',r'}(\Gammao))}-\int_{\Omega}[\nu\br\nabla
\bfu+\bbF]:\nabla\bfw\; \rmd\bfx \\
&=\ \blangle (\nu\br\nabla\bfu+\bbF-p_0\br\bbI)\cdot\bfn-\bfh,\bfw
\brangle_{\bfWp^{-1/r,r}(\Gammao),\bfWp^{1-1/r',r'}(\Gammao)}.
\end{align*}
The set of traces of all functions from $\cCs$ on $\Gammao$ is
dense in the set of all functions
$\bfw\in\bfWp^{1-1/r',r'}(\Gammao)$, such that
$\int_{\Gammao}\bfw\cdot\bfn\; \rmd l=0$. (This follows from the
density of the space of all functions
$\bfw\in\boldsymbol{\cC}_{\rm per}^{\infty}(\Gammao)$, such that
$\int_{\Gammao}\bfw\cdot\bfn\; \rmd l=0$, in
$\{\bfw\in\bfWp^{1-1/r',r'}(\Gammao);\
\int_{\Gammao}\bfw\cdot\bfn\; \rmd l=0\}$, and from the
possibility of extension of any function in $\boldsymbol{\cC}_{\rm
per}^{\infty}(\Gammao)$ from $\Gammao$ to $\Omega$ so that the
extended function is in $\cCs$.) Hence there exists $\cn04\in\R$
such that $\bfu$ and $p_0$ satisfy
\vspace{-5pt}
\begin{equation}
(-\nu\br\nabla\bfu-\bbF-p_0\br\bbI)\cdot\bfn-\bfh\ =\
\cc04\br\bfn, \label{4.7}
\end{equation}
as an equality in $\bfWp^{-1/r,r}(\Gammao)$. Put $p:=p_0-\cc04$.
Then $\bfu$ and $p$ satisfy the equation (\ref{1.1}) in the sense
of distributions in $\Omega$ and the boundary condition
(\ref{4.3}) as an equality in $\bfWp^{-1/r,r}(\Gammao)$. This
completes the proof of part 1). The statements in part 2) follow
from the next lemma.

Denote by $W^{1,r}_{\rm per}(\Omega)$ is the space of functions
from $W^{1,r}(\Omega)$, that satisfy in the sense of traces the
condition of periodicity on $\Gammam$ and $\Gammap$, analogous to
(\ref{1.5}).

\begin{lemma} \label{L3}
Let $\bff\in\bfL^r(\Omega)$ be given. Then there exists $\bbF\in
W^{1,r}_{\rm per} (\Omega)^{2\times 2}$, such that $\div\bbF=\bff$
a.e.~in $\Omega$, $\bbF\cdot\bfn=\bfzero$ a.e.~on $\Gammao$ in the
sense of traces and
\begin{equation}
\|\bbF\|_{1,r}\ \leq\ c\, \|\bff\|_r, \label{4.8}
\end{equation}
where $c=c(\Omega,r)$.
\end{lemma}

\begin{proof}
Denote $\Omegah:=\Omega\cup P$. Then $|\Omegah|=\tau d$. Define
$\bff:=\bfzero$ in $P$. Thus, $\bff\in\bfL^r(\Omegah)$. Put
$\bfk:=|\Omegah|^{-1}\, \int_{\Omegah}\bff\; \rmd\bfx$.

Let $\zeta=\zeta(x_1)$ be a smooth real function in $[0,d]$, such
that $\zeta(0)=1$ and $\zeta$ is supported in $[0,\delta]$, where
$\delta>0$ is so small that the profile $P$ (see Fig.~1) lies in
the stripe $\delta<x_1<d$. Since $\int_{\Omegah}\zeta'(x_1)\,
\rmd\bfx=-\tau$, we have $\int_{\Omegah}(\bff+\bfk d\,
\zeta'(x_1))\; \rmd\bfx=\bfzero$. Thus, due to \cite[Theorem
III.3.3]{Ga}, there exists $\bbF_0\in W^{1,r}_0(\Omegah)^{2\times
2}$, such that $\div\bbF_0=\bff+\bfk d\, \zeta'$ a.e.~in $\Omegah$
and
\begin{equation*}
\|\bbF_0\|_{1,r;\, \Omegah}\ \leq\ c\, \|\bff+\bfk d\,
\zeta'\|_{r;\, \Omegah}\ \leq\ c\, \|\bff\|_r,
\end{equation*}
where $c=c(\tau,d,\zeta)$. Since $\bfk d\, \zeta'(x_1)=\div[\bfk
d\, \zeta(x_1)\otimes \bfe_1]$, $\bbF_0$ satisfies
\begin{displaymath}
\div[\bbF_0-\bfk d\, \zeta(x_1)\otimes\bfe_1]\ =\ \bff
\end{displaymath}
a.e.~in $\Omega$. Put $\bbF:=\bbF_0-\bfk d\,
\zeta(x_1)\otimes\bfe_1$. The tensor function $\bbF$ has all the
properties, stated in the lemma.
\end{proof}

\vspace{4pt}
The proof of Theorem \ref{T2} is completed.
\end{proof}

\begin{remark} \label{R2} \rm
The identification of the right hand side of (\ref{1.1}) with
$\div\bbF$ enables us to deduce that $\bfu$ and $p$ satisfy
(\ref{4.3}) or (\ref{4.4}), which are equalities in
$\bfWp^{-1/r,r}(\Gammao)$.

If $\bff\in\bfW^{-1,2}_0(\Omega)$ and $\bbF$ is only in
$L^r(\Omega)^{2\times 2}$, as in part 1) of Theorem \ref{T2}, then
neither $(-\nu\br\nabla\bfu-p\br\bbI)\cdot\bfn$, nor
$\bbF\cdot\bfn$ need not be in $\bfW^{-1/r,r}_{\rm per}(\Gammao)$.
Thus, the sole term $\bbF\cdot\bfn$ need not have a sense on
$\Gammao$ and it is therefore generally not possible to require
$\bbF$ to satisfy $\bbF\cdot\bfn=\bfzero$ on $\Gammao$. The
situation is different if $\bff\in\bfL^r(\Omega)$ and $\bbF\in
W^{1,2}_{\rm per}(\Omega)$, see part 2) of Theorem \ref{T2}.
\end{remark}

In order to establish the existence of a weak solution of the
problem (\ref{1.1})--(\ref{1.6}), we shall need the next lemma. It
is a modification of a result, proven in \cite[Section 3]{FeNe3}.

\begin{lemma} \label{L2}
Let $1<r<\infty$ and $\bfg\in\bfW^{s,r}(\Gammai)$, where $s>1/r$
if $1<r\leq 2$ and $s=1-1/r$ if $r>2$. Let $\bfg$ satisfy the
condition $\bfg(\Am)=\bfg(\Ap)$. There exists a divergence--free
extension $\bfgs$ of $\bfg$ from $\Gammai$ to $\Omega$ and a
constant $\cn05>0$, independent of $\bfg$, such that
$\bfgs\in\bfW^{1,r}(\Omega)$, $\bfgs=\bfzero$ on $\Gammaw$ in the
sense of traces,

\vspace{4pt}
a) \ $\|\bfgs\|_{1,r}\leq\cc05\, \|\bfg\|_{s,r;\, \Gammai}$,

\vspace{4pt}
b) \ $\bfgs$ satisfies the condition of periodicity (\ref{1.5}) in
the sense of traces on $\Gammam\cup\Gammap$,

\vspace{4pt}
c) \ $\bfgs=(\Phi/\tau)\, \bfe_1$ \ in a neighbourhood of
$\Gammao$, where $\Phi=-\int_{\Gammai}\bfg\cdot\bfn\; \rmd l$.
\end{lemma}

\noindent
{\bf Principles of the proof.} \ The assumptions on number $s$
guarantee that $rs>1$ and it makes therefore sense to speak about
the traces of the function $\bfg$ at the end points $\Am$ and
$\Ap$ of $\Gammai$. This enables one to extend at first $\bfg$
from $\Gammai$ to $\partial\Omega$ in the way, described in
\cite[subsection 3.2]{FeNe1}. The extended function (which is
again denoted by $\bfg$) satisfies the condition of periodicity
(\ref{1.5}), the condition $\bfg=\bfzero$ on $\Gammaw$ and the
equality $\int_{\partial\Omega}\bfg\cdot\bfn\; \rmd l=0$. The
inequality $\|\bfg\|_{1-1/r,r;\, \partial\Omega}\leq c\,
\|\bfg\|_{s,r;\, \Gammai}$ can be proven by analogy with
\cite[Lemma 1]{FeNe1}, expressing the norm of $\bfg$ in
$\bfW^{s,r}(\Gammai)$ by formulas (2) and (3) in \cite{KuJoFu},
paragraph 8.3.2, p.~386, and the norm of $\bfg$ in
$\bfW^{1-1/r,r}(\partial\Omega)$ by formulas (II.4.8) and (II.4.9)
in \cite{Ga}, p.~64.

The existence of a divergence--free extension
$\bfgs\in\bfW^{1,r}(\Omega)$ of function $\bfg$ from
$\partial\Omega$ to $\Omega$ follows from \cite[Exercise
III.3.5]{Ga}. The extension satisfies the estimate
$\|\bfgs\|_{1,r}\leq c\, \|\bfg\|_{1-1/r,r;\, \partial\Omega}$
$\leq c\, \|\bfg\|_{s,r;\, \partial\Omega}$.

The function $\bfgs$ can be further modified in the way described
in \cite[Subsection 3.3]{FeNe3}, so that it finally has the
property c), too. Note that $r=2$ in \cite{FeNe3}. However,
dealing with general $r\in(1,\infty)$ does not affect the used
procedure and the resulting estimates. \hfill $\square$

\vspace{4pt}
The next theorem provides an information on the existence of a
weak solution to the Stokes problem (\ref{1.1})--(\ref{1.6}).

\begin{theorem}[on existence of a weak solution] \label{T3}
Let $1<r<\infty$, $\bff\in\bfW^{-1,r}_0(\Omega)$, $\bfg$ be a
given velocity profile on $\Gammai$ satisfying the assumptions of
Lemma \ref{L2}, and $\bfh\in\bfWp^{-1/r,r}(\Gammao)$. Then the
Stokes problem (\ref{1.1})--(\ref{1.6}) has a unique weak solution
$\bfu$ (in the sense of Definition \ref{D1}). The functions $\bfu$
and $p$ (an associated pressure, given by Theorem \ref{T2})
satisfy
\begin{equation}
\|\bfu\|_{1,r}+\|p\|_r\ \leq\ c\, \bigl[\br
\|\bff\|_{\bfW^{-1,r}_0} \|\bfg\|_{s,r;\, \Gammai}+
\|\bfh\|_{\bfWp^{-1/r,r}(\Gammao)} \bigr], \label{4.9}
\end{equation}
where $c=c(\Omega,\nu)$.
\end{theorem}

\begin{proof}
Note that if $\bfg$ satisfies the assumptions of Lemma \ref{L2}
then $\bfg$ also lies in $\bfWp^{1-1/r,r}(\Gammai)$.

Let $\bbF$ be a tensor function in $L^r(\Omega)^{2\times 2}$,
provided by Lemma \ref{L4} and $\bfF\in \Vsd{r}$ be a functional
in $\Vsd{r}$, defined by the formula
\vspace{-4pt}
\begin{equation}
\blangle\bfF,\bfw\brangle_{(\vsd{r},\vs{r'})}\ :=\
-\int_{\Omega}\bbF:\nabla\bfw\; \rmd\bfx \label{4.10}
\end{equation}

\vspace{-4pt} \noindent
for all $\bfw\in\Vs{r'}$. The norm of $\bfF$ in $\Vsd{r}$
satisfies the estimate $\|\bfF\|_{\vsd{r}}\leq c\, \|\bbF\|_r\leq
c\, \|\bff\|_{\bfW^{-1,r}_0}$. Similarly, let $\bfgs$ be the
function, given by Lemma \ref{L2}. Define a functional
$\bfG\in\Vsd{r}$ by the formula
\vspace{-4pt}
\begin{equation}
\blangle\bfG,\bfw\brangle_{(\vsd{r},\vs{r'})}\ :=\
-\int_{\Omega}\nabla\bfgs:\nabla\bfw\; \rmd\bfx \label{4.11}
\end{equation}

\vspace{-4pt} \noindent
for all $\bfw\in\Vs{r'}$. Then $\bfG$ satisfies
$\|\bfG\|_{\vsd{r}}\leq c\, \|\nabla\bfgs\|_r\leq c\,
\|\bfg\|_{s,r;\, \Gammai}$. Finally, let $\bfH$ be a functional in
$\Vsd{r}$, defined by the formula
\begin{equation}
\blangle\bfH,\bfw\brangle_{(\vsd{r},\vs{r'})}\ :=\ \langle\bfh,
\bfw\rangle_{(\bfWp^{-1/r,r}(\Gammao),\bfWp^{1-1/r',r'}(\Gammao))}
\label{4.12}
\end{equation}
for all $\bfw\in\Vs{r'}$. Obviously, $\bfH$ satisfies the estimate
$\|\bfH\|_{\vsd{r}}\leq c\, \|\bfh\|_{\bfW^{-1/r,r}(\Gammao)}$.

Due to Theorem \ref{T1}, the equation
\begin{equation}
\nu\cA_r\bfv\ =\ \bfF+\nu\bfG+\bfH \label{4.13}
\end{equation}
has a unique solution $\bfv\in\Vs{r}$, which satisfies
\begin{equation}
\|\nabla\bfv\|_r\ \leq\ c\, \bigl( \|\bfF\|_{\vsd{r}}+
\|\bfG\|_{\vsd{r}}+\|\bfH\|_{\vsd{r}} \bigr), \label{4.14}
\end{equation}
where $c=c(\nu,\Omega,r)$. The equation (\ref{4.13}) implies that
\begin{align*}
\nu\int_{\Omega}\nabla\bfv:\nabla\bfw\, \rmd\bfx\ =\
-\int_{\Omega}\bbF:\bfw\, \rmd\bfx-\nu\int_{\Omega}\nabla
\bfgs:\nabla\bfw\, \rmd\bfx+ \langle\bfh,
\bfw\rangle_{(\bfWp^{-1/r,r}(\Gammao),\bfWp^{1-1/r',r'}(\Gammao))}
\end{align*}
for all $\bfw\in\Vs{r'}$. Put $\bfu:=\bfv+\bfgs$. One can now
easily verify that $\bfu$ has all properties, stated in Definition
\ref{D1}, which means that $\bfu$ is a weak solution to the
problem (\ref{1.1})--(\ref{1.6}). The existence of an associated
pressure $p$ follows from Theorem \ref{T2}. The estimate
(\ref{4.9}) follows from (\ref{4.14}), from the estimates of the
norms of the functionals $\bfF$, $\bfG$ and $\bfH$ in $\Vsd{r}$
and from (\ref{4.6}) and (\ref{4.7}).
\end{proof}


\section{Proof of the inequality (\ref{3.4})} \label{S5}

Recall that in part (b) of the proof of Theorem \ref{T1}, we
assume that $\bbF\in W^{1,2}(\Omega)^{2\times 2}$ satisfies
(\ref{3.1}) and the functions
$\bfv\in\Vs{2}\cap\bfW^{2,2}(\Omega)$ and $p\in W^{1,2}(\Omega)$
satisfy the equation (\ref{3.3}) a.e.~in $\Omega$, the boundary
condition
\begin{equation}
\bfv\ =\ \bfzero \label{5.1}
\end{equation}
on $\Gammai\cup\Gammaw$ and the condition of periodicity
(\ref{1.5}) on $\Gammam$. It follows from \cite[Theorem 2]{TNe4}
that $\bfv$ and $p$ also satisfy the conditions of periodicity
\begin{align}
\frac{\partial\bfv}{\partial\bfn}(x_1,x_2+\tau)\ &=\
-\frac{\partial\bfv}{\partial\bfn}(x_1,x_2) &&
\mbox{for a.a.}\ \bfx\equiv(x_1,x_2)\in\Gammam, \label{5.2} \\
\noalign{\vskip 4pt}
p(x_1,x_2+\tau)\ &=\ p(x_1,x_2) && \mbox{fora.a.}\
\bfx\equiv(x_1,x_2)\in\Gammam \label{5.3} \\
\noalign{\noindent and the boundary condition}
\bigl(-\nu\nabla\bfv+p\br\bbI-\bbF\bigr)\cdot\bfn\ &\equiv\ -\nu\,
\frac{\partial\bfv}{\partial\bfn}+p\br\bfn-\bbF\cdot\bfn\ =\
\bfzero \quad && \mbox{a.e.~on}\ \Gammao. \label{5.4}
\end{align}
Extending $\bfv$, $p$ and $\bbF$ $\tau$--periodically in the
$x_2$--direction, we deduce that the extended functions (which we
again denote by $\bfv$, $p$ and $\bbF$) satisfy the equation
(\ref{3.3}) a.e.~in
\begin{displaymath}
\Omega_{\rm ext}\ :=\ \Omega\cup\Gammam^0\cup\Gammap^0\cup\bigl\{
\bfx=(x_1,x_2)\in\R^2;\ 0<x_1<d,\ \bfx\pm\tau\bfe_2
\in\Omega\bigr\}.
\end{displaymath}
The extended functions satisfy $\bfv\in\bfW^{2,2}(\Omega_{\rm
ext})$, $p\in W^{1,2}(\Omega_{\rm ext})$ and $\bbF\in
W^{1,2}(\Omega_{\rm ext})$. Note that
\begin{displaymath}
\partial\Omega_{\rm ext}\ =\ (\Gammai)_{\rm ext}\cup(\Gammao)_{\rm
ext}\cup(\Gammam-\tau\bfe_2)\cup(\Gammap+\tau\bfe_2)\cup(\Gammaw)_{\rm
ext},
\end{displaymath}
where $(\Gammai)_{\rm ext}$ is the part of the boundary of
$\Omega_{\rm ext}$ on $\gammai$, $(\Gammao)_{\rm ext}$ is the part
of the boundary of $\Omega_{\rm ext}$ on $\gammao$ and
\begin{displaymath}
(\Gammaw)_{\rm ext}\ :=\
\Gammaw\cup(\Gammaw+\tau\bfe_2)\cup(\Gammaw-\tau\bfe_2).
\end{displaymath}

We assume that $r>2$ throughout this section and we split the
derivation of the estimate (\ref{3.4}) into two parts, in which we
obtain the desired estimate in the interior of $\Omega_{\rm ext}$
plus a neighborhood of $\Gammai\cup\Gammaw$ (see Lemma \ref{L4.1})
and in a neighborhood of $\Gammao$ (see Lemma \ref{L4.2}).

\begin{lemma} \label{L4.1}
Let $\Omega'$ be a sub-domain of $\Omega_{\rm ext}$, such that the
distance between $\overline{\Omega'}$ and any of the curves
$\Gammam-\tau\bfe_2$, $\Gammap+\tau\bfe_2$, $\Gammaw-\tau\bfe_2$,
$\Gammaw+\tau\bfe_2$ and $(\Gammao)_{\rm ext}$ is positive. Then
$\bfv$ and $p$ satisfy the estimate
\begin{equation}
\|\bfv\|_{1,r;\, \Omega'}\ \leq\ c\, \|\bbF\|_r, \label{5.5}
\end{equation}
where $c=c(\nu,\Omega,\Omega')$.
\end{lemma}

\begin{proof}
We may assume, without loss of generality, that
$\Gammaw\subset\partial\Omega'$. There exists $\rho>0$ so small
that
\begin{displaymath}
U_{\rho}(\Omega')\ :=\ \bigl\{ \bfx=(x_1,x_2)\in\R^2\smallsetminus
P;\ 0<x_1<d,\ \dist(\bfx,\Omega')<\rho \bigr\}
\end{displaymath}
is a subset of $\Omega_{\rm ext}$. (Recall that $P$ is a compact
set, see Fig.~1.) There exists an infinitely differentiable
function $\eta$ in $\Omega_{\rm ext}$, such that $\eta=1$ in
$\Omega'$, $\supp\eta\subset U_{\rho}(\Omega')$ and
$\Omega'':=\supp\eta$ is of the class $C^2$. Denote
$\bfvt:=\eta\bfv$ and $\pt:=\eta p$. The functions $\bfvt$, $\pt$
represent a strong solution of the problem
\begin{align}
-\nu\Delta\bfvt+\nabla\pt\ &=\ \bfft && \mbox{in}\ \Omega'',
\label{5.6} \\
\div\bfvt\ &=\ \hht && \mbox{in}\ \Omega'', \label{5.7} \\
\bfvt\ &=\ \bfzero && \mbox{on}\ \partial\Omega'',
\label{5.8}
\end{align}
where
\begin{displaymath}
\bfft:=\eta\, \div\bbF-2\nu\br\nabla\eta\cdot\nabla\bfv-\nu\,
(\Delta\eta)\, \bfv-(\nabla\eta)\, p \qquad \mbox{and} \qquad
\hht:=\nabla\eta\cdot\bfv.
\end{displaymath}
Applying Proposition I.2.3 from \cite{Te}, p.~35, we obtain the
estimate
\begin{equation}
\|\bfvt\|_{1,r;\, \Omega''}\ \leq\ c\, \bigl(
\|\bfft\|_{\bfW^{-1,r}}+\|\hht\|_r \bigr). \label{5.9}
\end{equation}
Since $\bfL^2(\Omega)\hookrightarrow\bfW^{-1,r}(\Omega)$ and
$\bfW^{1,2}(\Omega)\hookrightarrow\bfL^r(\Omega)$, we can estimate
the terms on the right hand side as follows:
\begin{align*}
\|\bfft\|_{\bfW^{-1,r}}\ &\leq\ c\, \|\bbF\|_r+c\, \|\bfv\|_r+c\,
\|p\|_{-1,r}\ \leq\ \|\bbF\|_r+c\, \|\bfv\|_{1,2}+c\, \|p\|_2, \\
\|\hht\|_r\ &\leq\ c\, \|\bfv\|_r\ \leq\ c\, \|\bfv\|_{1,2}.
\end{align*}
Due to \cite[Lemma 1 and estimate (2.5)]{TNe4}, we have
$\|\bfv\|_{1,2}+\|p\|_2\leq c\, \|\bfF\|_{\vsd{2}}\leq c\,
\|\bbF\|_2\leq c\, \|\bbF\|_r$. Substituting these estimates to
(\ref{5.9}), we obtain: $\|\bfvt\|_{1,r;\, \Omega''}\leq c\,
\|\bbF\|_r$. Since $\Omega'\subset\Omega''$ and $\eta=1$ on
$\Omega'$, we obtain (\ref{5.5}).
\end{proof}

\begin{lemma} \label{L4.2}
Let $\Omega'$ be a sub-domain of $\Omega_{\rm ext}$, such that
$\overline{\Omega'}$ has a positive distance from any of the sets
$\Gammam-\tau\bfe_2$, $\Gammap+\tau\bfe_2$, $(\Gammai)_{\rm ext}$
and $(\Gammaw)_{\rm ext}$. Then $\bfv$ and $p$ satisfy the
estimate
\begin{equation}
\|\bfv\|_{1,r;\, \Omega'}\ \leq\ c\, \|\bbF\|_r, \label{5.10}
\end{equation}
where $c=c(\nu,\Omega,\Omega')$.
\end{lemma}

\begin{proof}
Since $\div\bfv=0$, there exists $\varphi\in W^{3,2}(\Omega_{\rm
ext})$, such that $\bfv\equiv(v_1,v_2)=\nabla^{\perp}\varphi
\equiv(-\partial_2 \varphi,\partial_1\varphi)$. As $\bfv=\bfzero$
on $(\Gammai)_{\rm ext}$, the function $\varphi$ satisfies
$\nabla^{\perp}\varphi=\bfzero$ on $(\Gammai)_{\rm ext}$. Since
$\varphi$ is determined uniquely up to an additive constant, we
can choose $\varphi$ so that
\begin{equation}
\varphi=0 \quad \mbox{and} \quad \nabla\varphi=\bfzero \qquad
\mbox{on}\ (\Gammai)_{\rm ext}. \label{5.11}
\end{equation}
Put
\begin{displaymath}
\bbZ\ :=\ -\nu\nabla\bfv+p\br\bbI-\bbF.
\end{displaymath}
Thus, if we denote by $Z_{ij}$ ($i,j=1,2$) the entries of $\bbZ$
and by $F_{ij}$ ($i,j=1,2$) the entries of $\bbF$, we can write
this formula in the form
\begin{align*}
\left( \begin{array}{cc} Z_{11}, & Z_{12} \\ Z_{21}, & Z_{22}
\end{array} \right)\ &=\ -\nu\, \left( \begin{array}{cc}
\partial_1v_1, & \partial_2 v_1 \\ \partial_1v_2, & \partial_2v_2
\end{array} \right)+\left( \begin{array}{cc} p, & 0 \\ 0, & p
\end{array} \right)-\left( \begin{array}{cc} F_{11}, & F_{12} \\
F_{21}, & F_{22} \end{array} \right) \\ \noalign{\vskip 2pt}
&=\ -\nu\, \left( \begin{array}{rr} -\partial_{12}\varphi, &
-\partial_{22}\varphi \\ \partial_{11}\varphi, & \partial_{21}
\varphi \end{array} \right)+\left( \begin{array}{cc} p, & 0 \\ 0,
& p \end{array} \right)- \left( \begin{array}{cc} F_{11}, & F_{12}
\\ F_{21}, & F_{22} \end{array} \right).
\end{align*}
Equation (\ref{3.3}) says that $\div\bbZ=\bfzero$, which means
that $\partial_jZ_{ij}=0$ for $i=1,2$. Hence there exist functions
$\psi_1,\, \psi_2\in W^{2,2}(\Omega_{\rm ext})$, such that
$Z_{11}=-\partial_2\psi_1$, $Z_{12}=\partial_1\psi_1$,
$Z_{21}=-\partial_2\psi_2$, $Z_{22}=\partial_1\psi_2$. Thus, we
obtain the equation
\begin{equation}
\left( \begin{array}{rr} -\partial_2\psi_1, & \partial_1\psi_1 \\
-\partial_2\psi_2, & \partial_1\psi_2 \end{array} \right)\ =\
-\nu\, \left( \begin{array}{rr} -\partial_{12}\varphi, &
-\partial_{22}\varphi \\ \partial_{11}\varphi, & \partial_{21}
\varphi \end{array} \right)+\left( \begin{array}{cc} p, & 0 \\ 0,
& p \end{array} \right)- \left( \begin{array}{cc} F_{11}, & F_{12}
\\ F_{21}, & F_{22} \end{array} \right). \label{5.12}
\end{equation}
This tensorial equation can also be considered to be a system of
four equations for four unknowns: $\varphi$, $p$, $\psi_1$ and
$\psi_2$. Since $\bfn=\bfe_1$ on $\Gammao$, the boundary condition
(\ref{5.4}) yields $Z_{11}=Z_{12}=0$ on $\Gammao$. This means that
$\partial_2\psi_1=\partial_2\psi_2=0$ on $\Gammao$, which implies
that $\psi_1$ and $\psi_2$ are constant on $\Gammao$. Let us
denote the constants by $k_1$ and $k_2$. Thus, we obtain the
boundary conditions
\begin{equation}
\psi_1=k_1, \quad \psi_2=k_2 \qquad \mbox{on}\ (\Gammao)_{\rm
ext}. \label{5.13}
\end{equation}
Let us concretely choose $k_1$ and $k_2$ so that
\begin{equation}
\int_{\Omega}\psi_1\; \rmd\bfx\ =\ \int_{\Omega}\psi_2\; \rmd\bfx\
=\ 0. \label{5.14}
\end{equation}
Later on in this proof, we shall need estimates of $\|\psi_1\|_r$
and $\|\psi_2\|_r$. Let us begin with $\|\psi_1\|_r$:
\begin{displaymath}
\|\psi_1\|_r\ =\ \sup_{w\in L^{r'}(\Omega);\ w\not=0}\
\|w\|_{r'}^{-1}\, \biggl|\int_{\Omega}\psi_1\br w\; \rmd\bfx
\biggr|.
\end{displaymath}
The function $w\in L^{r'}(\Omega)$ can be written in the form
$w=\overline{w}+w'$, where $\overline{w}:=\int_{\Omega}w\,
\rmd\bfx$ and $\int_{\Omega}w'\, \rmd\bfx=0$. Then, obviously,
$\|w'\|_{r'}\leq c\, \|w\|_{r'}$, where $c=c(\Omega)$. Moreover,
there exists $\bfz\in\bfW^{1,r}_0(\Omega)$, such that
$\div\bfz=w'$ and $\|\bfz\|_{1,r'}\leq c\, \|w'\|_{r'}$, see
\cite[Theorem III.3.3]{Ga}. Hence
\begin{align}
\|\psi_1\|_r\ &\leq\ c\, \sup_{w\in L^{r'}(\Omega);\ w\not=0}\
\|w'\|_{r'}^{-1}\, \biggl| \overline{w}\int_{\Omega}\psi_1\;
\rmd\bfx+\int_{\Omega}\psi_1\br w'\, \rmd\bfx \biggr| \nonumber \\
&\leq\ c\, \sup_{\bfz\in\bfW^{1,r'}_0(\Omega);\ \bfz\not=0}\
\|\nabla\bfz\|_{r'}^{-1}\, \biggl| \int_{\Omega}\psi_1\,
\div\bfz\; \rmd\bfx \biggr| \nonumber \\
&=\ c\, \sup_{\bfz\in\bfW^{1,r'}_0(\Omega);\ \bfz\not=0}\
\|\nabla\bfz\|_{r'}^{-1}\, \biggl|
\int_{\Omega}\nabla\psi_1\cdot\bfz\; \rmd\bfx \biggr| \nonumber \\
\noalign{\vskip 2pt}
&=\ c\, \|\nabla\psi_1\|_{\bfW^{-1,r}}\ \leq\ c\,
\|\bbZ\|_{\bfW^{-1,r}}\ =\ c\, \bigl\|
-\nu\nabla\bfv+p\br\bbI-\bbF \bigr\|_{\bfW^{-1,r}} \nonumber \\
\noalign{\vskip 2pt}
&\leq\ c\, \bigl( \|\bfv\|_r+\|p\|_{\bfW^{-1,r}}+\|\bbF\|_r
\bigr). \label{5.15}
\end{align}
As $\bfW^{1,r'}_0(\Omega)\hookrightarrow\bfL^2(\Omega)$, we also
have $\bfL^2(\Omega)\hookrightarrow\bfW^{-1,r}(\Omega)$. Hence
$\|p\|_{\bfW^{-1,r}}\leq c\, \|p\|_2$. Furthermore, as
$\bfW^{1,2}(\Omega)\hookrightarrow\bfL^r(\Omega)$, we also have
$\|\bfv\|_r\leq c\, \|\bfv\|_{1,2}$. Thus, (\ref{5.15}) yields
\begin{equation}
\|\psi_1\|_r\ \leq\ c\, \bigl( \|\bfv\|_{1,2}+\|p\|_2+\|\bbF\|_r
\bigr). \label{5.16}
\end{equation}
Due to \cite[Lemma 1 and estimate (2.5)]{TNe4}, the right hand
side is less than or equal to $c\, \bigl(
\|\bfF\|_{\vsd{2}}+\|\bbF\|_r \bigr)\ \leq\ c\, \bigl( \|\bbF\|_2+
\|\bbF\|_r \bigr)$, which is less than or equal to $c\,
\|\bbF\|_r$. Substituting this to (\ref{5.16}) and taking into
account that $\psi_2$ can be estimated in the same way, we finally
obtain
\begin{equation}
\|\psi_1\|_r+\|\psi_2\|_r\ \leq\ c\, \|\bbF\|_r. \label{5.17}
\end{equation}

In order to obtain the desired estimate (\ref{3.4}), we will apply
Theorem 10.2 from the paper \cite{AgDoNi} by S.~Agmon. A.~Douglis
and L.~Nirenberg. This theorem, however, concerns a boundary value
problem in a half-space, which means a half-plane in the case of a
problem in 2D. In order to transform the problem (\ref{5.12}),
(\ref{5.13}) to a problem in the half-plane $x_1<d$, we apply the
cut--off function technique; there exists $\rho>0$ so small that
\begin{displaymath}
U_{\rho}(\Omega')\ :=\ \bigl\{\bfx=(x_1,x_2)\in\R^2;\ x_1<d,\
\dist(\bfx,\Omega')<\rho \bigr\}
\end{displaymath}
is a subset of $\Omega_{\rm ext}$. Let $\eta$ be an infinitely
differentiable function in $\Omega_{\rm ext}$, such that $\eta=1$
in $\Omega'$ and $\eta=0$ in $\Omega_{\rm ext}\smallsetminus
U_{\rho}(\Omega')$. Multiplying (\ref{5.12}) by $\eta$, and
denoting $\varphit:=\eta\varphi$, $\psit_1:=\eta\psi_1$,
$\psit_2:=\eta\psi_2$ and $\pt:=\eta p$, we get
\begin{align}
& \left( \begin{array}{rr} -\partial_2\psit_1, & \partial_1\psit_1 \\
-\partial_2\psit_2, & \partial_1\psit_2 \end{array} \right)+ \nu\,
\left( \begin{array}{rr} -\partial_{12}\varphit, &
-\partial_{22}\varphit \\ \partial_{11}\varphit, & \partial_{21}
\varphit \end{array} \right)-\left( \begin{array}{cc} \pt, & 0 \\
0, & \pt \end{array} \right) \nonumber \\
& \hspace{10pt} =\ -\left( \begin{array}{cc} \eta F_{11}, & \eta
F_{12} \\ [1pt] \eta F_{21}, & \eta F_{22} \end{array} \right)+
\left( \begin{array}{rr} -(\partial_2\eta)\, \psi_1, &
(\partial_1\eta)\, \psi_1  \\ [1pt] -(\partial_2\eta)\, \psi_2, &
(\partial_1\eta)\, \psi_2 \end{array} \right)+\left(
\begin{array}{rr} -(\partial_{12}\eta)\, \varphi, &
-(\partial_{22}\eta)\, \varphi \\ (\partial_{11}\eta)\, \varphi, &
(\partial_{21}\eta)\, \varphi \end{array} \right) \nonumber \\
& \hspace{25pt} + \left( \begin{array}{rr}
-(\partial_1\eta)(\partial_2\varphi)-(\partial_2\eta)(\partial_1\varphi),
& -2\br (\partial_2\eta)(\partial_2\varphi) \\ [1pt]
2\br(\partial_1\eta)(\partial_1\varphi), &
(\partial_1\eta)(\partial_2\varphi)+(\partial_2\eta)
(\partial_1\varphi) \end{array} \right). \label{5.18}
\end{align}
Since $\eta$ is supported in the closure of $U_{\rho}(\Omega')$,
we can treat (\ref{5.18}) as a system of four equations in the
half-plane $x_1<d$ for the unknowns $\varphit$, $\psit_1$,
$\psit_2$ and $\pt$ with the boundary conditions
\begin{equation}
\psit_1=\eta\br k_1, \quad \psit_2=\eta\br k_2 \qquad \mbox{on}\
\gammao. \label{5.19}
\end{equation}
In order to simplify the equation (\ref{5.18}) and to use a
notation, consistent with \cite{AgDoNi}, we eliminate $\pt$ by
subtracting the two corresponding equations and denote
$u_1:=\varphit$, $u_2:=\psit_1$ and $u_3:=\psit_2$. Then the
problem (\ref{5.18}), (\ref{5.19})  reduces to the system of three
equations
\begin{align}
2\nu\, \partial_{12}u_1+\partial_2u_2+\partial_1u_3\ &=\
-\eta\br(F_{22}-F_{11})+(\partial_1\eta)\br\psi_2+
(\partial_2\eta)\br\psi_1+2\br(\partial_{12}\eta)\br\varphi
\nonumber \\
& \hspace{15pt} +2\br(\partial_1\eta)(\partial_2\varphi)+
2\br(\partial_2\eta) (\partial_1\varphi), \label{5.20} \\
\noalign{\vskip 2pt}
\nu\, \partial_{11}u_1-\partial_2u_3\ &=\ -\eta
F_{21}-(\partial_2\eta)\br\psi_2+(\partial_{11}\eta)\br\varphi+
2\br(\partial_1\eta)(\partial_1\varphi), \label{5.21} \\
\noalign{\vskip 2pt}
\nu\, \partial_{22}u_1-\partial_1u_2\ &=\ -\eta
F_{12}-(\partial_1\eta)\br\psi_1+(\partial_{22}\eta)\br\varphi+
2\br(\partial_2\eta)(\partial_2\varphi) \label{5.22}
\end{align}
with the boundary conditions
\begin{equation}
u_2=\eta\br k_1, \quad u_3=\eta\br k_2 \qquad \mbox{on}\ \gammao.
\label{5.23}
\end{equation}
Obviously, the leading differential operator in
(\ref{5.20})--(\ref{5.22}) is
\begin{displaymath}
\operL(\partial_1,\partial_2)\ :=\ \left( \begin{array}{ccc}
2\nu\br\partial_1\partial_2, & \partial_2, & \partial_1, \\
\nu\br\partial_1^2, & 0, & -\partial_2, \\ \nu\br\partial_2^2, &
-\partial_1, & 0 \end{array} \right).
\end{displaymath}
The determinant of $\operL(\xi_1,\xi_2)$ equals $-\nu\,
(\xi_1^2+\xi_2^2)^2$, which is a polynomial of degree $2m=4$. With
this information, one can verify that the system
(\ref{5.20})--(\ref{5.22}) is ``uniformly elliptic'' of order $4$
in the sense of Agmon--Douglis--Nirenberg, which means that it
satisfies the conditions (1.1)--(1.7) from \cite{AgDoNi}, pp.~38,
39, and the so called ``supplementary condition'', see
\cite[p.~39]{AgDoNi}. Moreover, the boundary conditions
(\ref{5.23}), the number of which is $m=2$, satisfy condition
(2.2) (see \cite[p.~42]{AgDoNi}) and have all required properties
of the so called ``complementing boundary conditions'', formulated
in \cite{AgDoNi}, p.~42--43. The verification is elementary,
however technical, hence we do not provide the details here.

Note, that one can also find the description and explanation of
the conditions that enable one to call a considered system
``uniformly elliptic'' and considered boundary conditions to be
``complementing'' in \cite[Appendix D] {BoGu}.

Denote by $\R^2_{d-}$ the half-plane $x_1<d$.

The inclusions $u_1\in W^{3,2}(\R^2_{d-})\hookrightarrow
W^{2,r}(\R^2_{d-})$, $u_2,u_3\in W^{2,2}(\R^2_{d-})\hookrightarrow
W^{1,r}(\R^2_{d-})$ and the verification of the aforementioned
conditions from \cite{AgDoNi} enable us to apply Theorem 10.2 from
\cite{AgDoNi}, which yields the estimate
\begin{align}
\|u_1\|_{2,r;\, \R^2_{d-}} & +\|u_2\|_{1,r;\,
\R^2_{d-}}+\|u_3\|_{1,r;\, \R^2_{d-}} \nonumber \\
&\leq\ c\, \sum_{i=1}^3 \|\ft_i\|_{r;\, \R^2_{d-}}+c\, \bigl(
\|\eta\br k_1\|_{1-1/r,r;\,
\gammao}+\|\eta\br k_2\|_{1-1/r,r;\, \gammao} \bigr) \nonumber \\
&\leq\ c\, \sum_{i=1}^3 \|\ft_i\|_{r;\, \R^2_{d-}}+c\, \bigl(
|k_1|+|k_2| \bigr), \label{5.24}
\end{align}
where $\ft_1,\, \ft_2$ and $\ft_3$ denote the right hand sides of
equations (\ref{5.20}), (\ref{5.21}) and (\ref{5.22}),
respectively, and $c$ is independent of $u_i$, $\ft_i$ ($i=1,2,3$)
and $k_1$, $k_2$. The first term on the right hand side of
(\ref{5.24}) can be estimated by means of the inequalities
\begin{align}
\sum_{i=1}^3\|\ft_i\|_{r;\, \R^2_{d-}}\ &\leq\ c\, \bigl(
\|\bbF\|_r+\|\psi_1\|_r+\|\psi_2\|_r+\|\varphi\|_{1,r} \bigr).
\label{5.25}
\end{align}
The norms $\|\psi_1\|_r$ and $\|\psi_2\|_r$ can be estimated by
means of (\ref{5.17}). The norm $\|\varphi\|_{1,r}$ can be
estimated as follows:
\begin{equation}
\|\varphi\|_{1,r}\ \leq\ c\, \|\bfv\|_r\ \leq\ c\, \|\bfv\|_{1,2}\
\leq\ c\, \|\bfF\|_{\vsd{2}}\ \leq\ c\, \|\bbF\|_2\ \leq\ c\,
\|\bbF\|_r. \label{5.26}
\end{equation}
(Here, the first inequality holds due to the definition of
$\varphi$ and (\ref{5.11}), the second inequality follows from the
continuous imbedding $\bfW^{1,2}(\Omega)\hookrightarrow
\bfL^r(\Omega)$, the third inequality follows from \cite[Lemma
1]{TNe4}, the fourth inequality follows from the definition of
$\bfF$ and the last inequality holds, because $r>2$.) The second
term on the right hand side of (\ref{5.24}) can be estimated by
means of these inequalities:
\begin{align}
|k_1|+|k_2|\ &=\ |\Omega|^{-1}\, \biggl|
\int_{\Omega}(\psi_1-k_1)\; \rmd\bfx \biggr|+|\Omega|^{-1}\,
\biggl| \int_{\Omega}(\psi_2-k_2)\; \rmd\bfx \biggr| \nonumber \\
&\leq\ c\, \|\nabla(\psi_1-k_1)\|_2+c\, \|\nabla(\psi_2-k_2)\|_2\
\leq\ c\, \|\bbZ\|_2\ \leq\ c\, \bigl( \|\nabla\bfv\|_2+\|p\|_2+
\|\bbF\|_2 \bigr) \nonumber \\ \noalign{\vskip 2pt}
&\leq\ c\, \|\bfF\|_{\vsd{2}}+c\, \|\bbF\|_2\ \leq\ c\,
\|\bbF\|_2\ \leq\ c\, \|\bbF\|_r. \label{5.27}
\end{align}
(The inequalities $\bigl| \int_{\Omega} (\psi_j-k_j)\,
\rmd\bfx\bigr|\leq c\, \|\nabla(\psi_j-k_j)\|_2$ (for $j=1,2$)
hold due to (\ref{5.13}). The first inequality on the last line
follows from Lemma 1 and estimate (2.5) in \cite{TNe4} and the
last estimate follows from the condition $r>2$.) Estimates
(\ref{5.24})--(\ref{5.27}) now yield
\begin{align*}
\|\varphi\|_{2,r;\, \Omega'}\ &\leq\ \|\varphit\|_{2,r;\,
\R^2_{d-}}\ =\ \|u_1\|_{2,r;\, \R^2_{d-}}\ \leq\ c\, \|\bbF\|_r
\end{align*}
Since $\|\bfv\|_{1,r;\, \Omega'}\leq c\, \|\varphi\|_{2,r;\,
\Omega'}$, we obtain inequality (\ref{5.10}).
\end{proof}

\vspace{4pt}
Now the estimate (\ref{3.4}) follows easily from Lemmas \ref{L4.1}
and \ref{L4.2}.

\medskip \noindent
{\bf Acknowledgement.} This work was supported by European
Regional Development Fund-Project ``Center for Advanced Applied
Science'' No.~CZ.02.1.01/0.0/0.0/16\_019/0000778.

\bigskip \medskip
\hspace{1.2pt} \begin{tabular}{ll} Author's address: \quad &
Tom\'a\v{s} Neustupa \\ & Czech Technical University \\ & Faculty
of
Mechanical Engineering \\ & Department of Technical Mathematics \\
& Karlovo n\'am.~13, 121 35 Praha 2 \\ & Czech Republic \\ &
e-mail: \ tomas.neustupa@fs.cvut.cz
\end{tabular}

\end{document}